\documentclass{amsart}

\usepackage[utf8]{inputenc}

\usepackage[totalwidth=339.7pt, totalheight=537pt]{geometry}

\usepackage[usenames,dvipsnames]{color}
\usepackage{amsthm,amsfonts,amssymb,amsmath}
\usepackage[all]{xy}
\usepackage{xr-hyper}
\usepackage[colorlinks=
   citecolor=Black,
   linkcolor=Red,
   urlcolor=Blue]{hyperref}
\usepackage{verbatim}
\usepackage{pdfsync}

\newcommand{\BA}{{\mathbb {A}}}

\newcommand{\BC}{{\mathbb {C}}}
\newcommand{\BD}{{\mathbb {D}}}

\newcommand{\BF}{{\mathbb {F}}}
\newcommand{\BG}{{\mathbb {G}}}

\newcommand{\BM}{{\mathbb {M}}}
\newcommand{\BN}{{\mathbb {N}}}

\newcommand{\BP}{{\mathbb {P}}}
\newcommand{\BQ}{{\mathbb {Q}}}
\newcommand{\BR}{{\mathbb {R}}}
\newcommand{\BS}{{\mathbb {S}}}

\newcommand{\BX}{{\mathbb {X}}}

\newcommand{\BZ}{{\mathbb {Z}}}

\newcommand{\CD}{{\mathcal {D}}}
\newcommand{\CE}{{\mathcal {E}}}
\newcommand{\CF}{{\mathcal {F}}}
\newcommand{\CG}{{\mathcal {G}}}

\newcommand{\CI}{{\mathcal {I}}}

\newcommand{\CL}{{\mathcal {L}}}
\newcommand{\CM}{{\mathcal {M}}}

\newcommand{\CO}{{\mathcal {O}}}

\newcommand{\CT}{{\mathcal {T}}}

\newcommand{\CV}{{\mathcal {V}}}

\newcommand{\End}{{\mathrm{End}}}

\newcommand{\Gal}{{\mathrm{Gal}}}
\newcommand{\GL}{{\mathrm{GL}}}

\newcommand{\Hom}{{\mathrm{Hom}}}

\newcommand{\Ker}{{\mathrm{Ker}}\,}

\newcommand{\Lie}{{\mathrm{Lie}}\,}

\newcommand{\Res}{{\mathrm{Res}}}

\newcommand{\SL}{{\mathrm{SL}}}
\newcommand{\Spec}{{\mathrm{Spec}}\,}

\newcommand{\Sp}{{\mathrm{Sp}}}

\newcommand{\lra}{\longrightarrow}

\def\dom{{\rm dom}}

\def\id{{\rm id}}

\def\rig{{\rm rig}}

\def\Res{{\rm Res}}
\def\Lie{{\rm Lie}}

\def\Hom{{\rm Hom}}

\def\Spf{{\rm Spf}}
\def\Spec{{\rm Spec}}

\def\cls{{\rm cls}}
\def\der{{\rm der}}
\def\tor{{\rm tor}}

\newtheorem{theorem}{Theorem}[section]
\newtheorem{proposition}[theorem]{Proposition}
\newtheorem{lemma}[theorem]{Lemma}
\newtheorem {conjecture}[theorem]{Conjecture}
\newtheorem{corollary}[theorem]{Corollary}
\newtheorem {question}[theorem]{Question}
\newtheorem {properties}[theorem]{Properties}
\theoremstyle{definition}
\newtheorem{definition}[theorem]{Definition}
\newtheorem{example}[theorem]{Example}

\newtheorem{remark}[theorem]{Remark}
\newtheorem{remarks}[theorem]{Remarks}

\numberwithin{equation}{section}

\newenvironment{altenumerate}
   {\begin{list}
      {\textup{(\theenumi)} }
      {\usecounter{enumi}
       \setlength{\labelwidth}{0pt}
       \setlength{\labelsep}{0pt}
       \setlength{\leftmargin}{0pt}
       \setlength{\itemsep}{\the\smallskipamount}
       \renewcommand{\theenumi}{\roman{enumi}}
      }}
   {\end{list}}
\newenvironment{altitemize}
   {\begin{list}
      {$\bullet~~$ }
      {\setlength{\labelwidth}{0pt}
       \setlength{\labelsep}{0pt}
       \setlength{\leftmargin}{0pt}
       \setlength{\itemsep}{\the\smallskipamount}
      }}
   {\end{list}}

\begin{document}
%------------------------------------------------------

\begin{abstract}
This is a survey article that advertizes the idea that  there should exist a theory of $p$-adic local analogues of Shimura varieties. Prime examples are the towers of rigid-analytic spaces defined by Rapoport-Zink spaces, and we  also review their theory  in the light of this idea.   We also discuss conjectures on the $\ell$-adic cohomology of local Shimura varieties.
\end{abstract}

\thanks{The research of Rapoport was partially supported by SFB/TR 45 ``Periods, Moduli Spaces and 
Arithmetic of Algebraic Varieties" of the DFG. The research of Viehmann was partially supported by ERC starting grant 277889.}
\title{Towards a theory of local Shimura varieties}
\author{ Michael Rapoport  and Eva Viehmann}

\date{\today}
\maketitle
\centerline{\it To P. Schneider on his 60th birthday}
\tableofcontents
\section{Introduction} 

Deligne's definition of a Shimura variety \cite{Del-Bourb} is of ultimate elegance. Starting from a {\it Shimura datum}, i.e.,
a pair $(G,\{ h\})$, where $G$ is a reductive algebraic group over $\BQ$ and $\{ h\}$ is a $G(\BR)$-conjugacy class
 of homomorphism $h : \BS \lra G_{\BR}$ satisfying three simple axioms, Deligne defines a
tower of complex-analytic spaces $\{ \BM^K \mid K \subset G(\BA_f) \}$. Here $K$ runs through the compact
open subgroups of $G(\BA_f)$, and $G(\BA_f)$ acts on this tower as a group of Hecke correspondences.
The individual members of this tower are defined as
\begin{equation*}
 \BM^K = G (\BQ) \setminus [ X_{\{h\}} \times G(\BA_f)/K ] \, ,
\end{equation*}
where $X_{\{h\}}$ is the complex variety of elements $h$ of $\{h\}$. It is well-known that $\BM^K$ has a
unique structure of quasi-projective variety over $\BC$, compatible with all transition maps in this tower.
In fact, the tower is equipped with an effective descent from $\BC$ to the {\it Shimura field} $E=E(G,\{h\})$
and this descended model can be characterized as a {\it canonical model} in the sense of Deligne \cite{Del-Bourb}.

Fix a prime number $p$. The purpose of the present paper is to advertize the idea that a similar theory should exist over $\BQ_p$ 
instead of $\BQ$. Starting from a {\it local Shimura datum}, i.e., a triple $(G, [b], {\{\mu\}})$ consisting of
a reductive algebraic group $G$ over $\BQ_p$, a $\sigma$-conjugacy class $[b]$ of elements in $G(\breve\BQ_p)$,
and a conjugacy class ${\{\mu\}}$ of cocharacters of $G_{\bar{\BQ}_p}$, satisfying some simple axioms, we would
like to associate to it a {\it local Shimura variety}, i.e., a tower of rigid-analytic spaces $\{ \BM^{K} \mid K \subset G (\BQ_p) \}$ over $\breve E$.  Here $E$ denotes the {\it local Shimura field} $E$ associated to $(G, \{\mu\})$ and for any $p$-adic local field $F$ we denote by $\breve F$ the completion of the maximal unramified extension of $F$. The group $G(\BQ_p)$ acts on this tower via Hecke correspondences. Furthermore, and this is at variance with the classical theory of Shimura varieties, the $\sigma$-centralizer group $J(\BQ_p)$ acts on each member of the tower in a compatible manner. The members of this tower should all map to the period domain $\breve\CF(G, b, \{\mu\})$ associated to $(G, [b], {\{\mu\}})$, in the sense of \cite{DOR},   in a compatible way, and equivariantly for the action of all groups in play. Furthermore, there should be a {\it Weil descent datum} on this tower from $\breve E$ down to $E$. 

The prototype of such a tower is given by the tower 
defined by a Rapoport-Zink space. Recall that this tower arises in two steps: in a  first step, one constructs a formal scheme over $\Spf\, \CO_{\breve E}$; in a second step, one takes its generic fiber and constructs the tower over it. Here the formal scheme arises as the solution of a moduli problem of $p$-divisible groups, and the tower arises by considering level structures on the universal $p$-divisible group over the generic fiber. The thrust of the present paper is that it should be possible to construct the tower directly. Then the formal schemes constructed in \cite{RZ} would be considered as providing {\it integral models} of {\it certain} members of the tower, just as this happens in the case of  classical Shimura varieties. 

At present we are not able to define such a tower, except in cases closely
related to RZ-spaces. However, we strongly believe in their existence.  Besides aesthetic reasons and the analogy with the global theory of Shimura varieties, there are also systematic reasons to believe that such a
construction exists. And the prize to be won is simply too big not to search for such a construction. The $\ell$-adic cohomology of the tower admits a simultaneous action of $G(\BQ_p)$, of the Weil group $W_E$ and of $J(\BQ_p)$. The actions of these three groups are intertwined in a deep way, and often define local Langlands correspondences in a geometric way.  A theory of local Shimura varieties should provide enough flexibility to establish such geometric Langlands correspondences in many cases. In the text, we discuss in this context two cohomology conjectures:
1) the Kottwitz conjecture which concerns  the discrete part of the cohomology, when $[b]$ is a {\it basic} $\sigma$-conjugacy class in the sense of 
Kottwitz \cite{K2},
2) the Harris conjecture which gives an inductive formula for the cohomology when $[b]$ is not basic. 

The point of view of local Shimura varieties also opens a new perspective on the ``classical case'' of RZ-spaces, and this  raises a number of questions concerning them, which seem interesting in their own right. 
The second aim of this survey paper is thus, besides
advocating the idea of general local Shimura varieties, to formulate a number of questions and conjectures on RZ-spaces that arise in this way. It turns out that, by adopting a rigorously group-theoretic 
set-up, the theory of RZ-spaces is stripped of its origins as the solution of formal moduli problems, and 
 becomes more streamlined. 
 
 None of these problems are solved here. It is not clear  how much   the methods that have been 
developed over the last years to analyze the formal schemes of \cite{RZ} can contribute to solve them. Encouraging in this direction is the recent paper by W.~Kim \cite{Kim} in which he enlarges substantially the class of RZ-spaces. Also, the construction by Kisin \cite{Ki} of integral models of Shimura varieties of abelian type should yield  new spaces of RZ-type by global methods, via the uniformization theorem of \cite{RZ}. However, neither of these constructions is in the spirit of this paper. More in line with the ideas outlined here is  Scholze's theory of perfectoid spaces \cite{Scholze-perf, Scholze-surv}. Indeed, this is  a theory that takes place largely in the generic fiber, and our hope is that these novel methods (perhaps combined with the theory of Fargues/Fontaine of ``the curve'' in a relative setting) can advance the theory that we have in mind. In fact, this hope is already in the process of being put into reality: as a first step in this direction, Scholze and Weinstein \cite{Scholzewein} show in the case of a local Shimura datum $(G, [b], {\{\mu\}})$ with $G=\GL_n$ that the inverse limit 
in the corresponding RZ tower represents a moduli problem that does not involve $p$-divisible groups, and it seems quite likely that, using this description of the inverse limit, one can construct local Shimura varieties for local Shimura data of {\it local Hodge type}, i.e., those that admit an embedding into a local Shimura datum for $\GL_n$. 

Many of the questions raised here are already present, at least in a tentative form, in the ICM talk \cite{Rapo-ICM} of one of us (M. R.).  However, in the intervening time there has been enormous progress in this area. For instance, we have a much better understanding of the covering spaces of open subspaces of period domains, through the work of Faltings, Fargues, Hartl and Scholze. Also, the cohomology of RZ spaces has been much elucidated, through the work of Boyer, Dat, Fargues, Mantovan, Mieda, Strauch, Shin  and others. We also understand better the situation in those cases when the RZ data involve ramification. For instance, we have now a purely group-theoretical definition of  local models of Shimura varieties through the work of Pappas and Zhu \cite{PZ}. This result is related to the concrete construction of local models through moduli problems just as the putative theory of local Shimura varieties is related to the concrete constructions of them through RZ spaces.

Here is a section-by-section description of the contents of the paper. In section \ref{sigmaconjclas}, we recall a few facts from Kottwitz' s theory of $\sigma$-conjugacy classes,  and, in particular,  introduce, for a given geometric conjugacy class $\{\mu\}$ of cocharacters of $G$ the finite set $A(G, \{\mu\})$ of {\it acceptable } $\sigma$-conjugacy classes, and the subset $B(G, \{\mu\})\subseteq A(G, \{\mu\})$ of {\it neutral acceptable} $\sigma$-conjugacy classes. Section \ref{perdom} is a very brief reminder of the definition of {\it period domains}. In section \ref{rzspaces}, we discuss RZ {\it data} and their associated formal schemes; we also define the RZ {\it tower} and discuss various questions that arise in this context, when viewing this tower as a typical example of a local Shimura variety. In section \ref{locShvar} we explain the conjectural concept of a local Shimura variety, and give a few examples beyond those coming from RZ spaces. In section \ref{elladiccoho}, we define the $\ell$-adic cohomology modules defined by a local Shimura variety. In section \ref{seckottwitzconj}, we discuss the cohomology conjecture of Kottwitz, and in section \ref{harrisconj} the cohomology conjecture of Harris. 

 This paper was triggered by two concurrent events: the fact that one of us (E. V.) noted that the Harris conjecture had to be modified in a non-trivial way, and the ETH lectures 2011 of one of us (M. R.) which offered an opportunity to think through again some  aspects of the theory in question in the light of recent developments in this area. 
 
 The results and ideas of L.~Fargues and P.~Scholze were all-important in the recent development of the ideas presented here; in addition, we profited greatly from their generous advice and their comments. We also thank J.-F.~Dat, U.~G\"ortz, X.~He, T.~Kaletha, Y.~Mieda and X.~Shen for helpful remarks. We also thank the referee for his remarks. 
   
\smallskip

{\bf Notation.} Throughout the paper, a prime number $p$ is fixed. For  a $p$-adic local field $F$, we denote by $\breve F$ the completion of the maximal unramified extension $F^{{\rm un}}$ (in a fixed algebraic closure $\overline F$ of $F$), and by $\sigma\in \Gal(\breve F/F)$ the relative Frobenius automorphism. For a reductive algebraic group $G$, we denote by $G_{\rm der}$ its derived group, by $G_{\rm sc}$  the simply connected cover group of $G_{\rm der}$, and by $G_{\rm ab}$ the maximal torus quotient of $G$.

\section{$\sigma$-conjugacy classes }\label{sigmaconjclas}

In this section, we recall Kottwitz's theory of $\sigma$-conjugacy classes and single out those that are {\it acceptable} wrt a chosen conjugacy class of cocharacters.

\subsection{Definitions}Let $G$ be a connected reductive algebraic group over the finite extension $F$ of $\BQ_p$.  We denote by $B(G)=B(F, G)$ the set of equivalence classes of $G(\breve F)$ by the equivalence relation of being $\sigma$-conjugate, cf. \cite{Kottwitz1}, 
\begin{equation*}
b\sim b'\iff b'=g^{-1} b\sigma(g),\quad  g\in G(\breve F) . 
\end{equation*}
For $b\in G(\breve F)$ we denote $[b]=\{g^{-1}b\sigma(g)\mid g\in G(\breve F)\}$ its $\sigma$-conjugacy class.
 
Kottwitz, cf. \cite{Kottwitz1}, has given a classification of the set $B(G)$ (compare also \cite{RapoportRichartz}, Section 1). This classification  generalizes the Dieudonn\'e-Manin classification of isocrystals by their Newton polygons if $F=\BQ_p$. For us the following version of Kottwitz's results is convenient, cf. \cite{RapoportRichartz}.
  
A $\sigma$-conjugacy class is determined by two invariants. One of them is given by the image of $[b]$ under the {\it Kottwitz map} 
  \begin{equation}
  \kappa_G:B(G)\rightarrow \pi_1(G)_{\Gamma} .
  \end{equation}  Here $\pi_1(G)$ denotes Borovoi's algebraic fundamental group  of $G$, \cite{Borovoi}, and we are taking the coinvariants under the action of the absolute Galois group $\Gamma$ of $F$. The second invariant is obtained by assigning to $b$ a homomorphism $\tilde\nu_b:\mathbb{D}\rightarrow G_{\breve F}$ where $\mathbb{D}$ is the pro-algebraic torus whose character group is $\mathbb{Q}$.  The conjugacy class of this homomorphism is then $\sigma$-stable. Let $X_*(G)_\BQ/G$ be the set of conjugacy classes of homomorphisms $\BD_{\breve F}\to G_{\breve F}$. We denote by  
  \begin{equation}
  \nu_{[b]}\in\big(X_*(G)_\BQ/G\big)^{\langle\sigma\rangle}
  \end{equation}
   the conjugacy class of $\tilde\nu_b$. This yields the second classifying invariant  of $[b]$, and  is called the {\it Newton point} associated with $[b]$; the map $[b]\mapsto \nu_{[b]}$ is called the {\it Newton map}.  The images of $\nu_{[b]}$ and $\kappa_G([b])$ in $\pi_1(G)_{\Gamma}\otimes \mathbb{Q}$ coincide, cf. \cite{RapoportRichartz}, Thm. 1.15\footnote{It has been pointed out to us that the reference to \cite{Kottwitz1} in loc.~cit. may be insufficiently detailed. However, a proof of this statement follows the standard strategy of Kottwitz: for tori, the statement is \cite{Kottwitz1}, 4.4 and 2.8; the general case 
   follows from  the case of tori,  via functoriality and the identification $\pi_1(G)_{\Gamma}\otimes \mathbb{Q}=\pi_1(G_{\rm ab})_{\Gamma}\otimes \mathbb{Q}$. Compare also the proof of \cite{Kottwitz-glob}, Prop.~11.4.}. Both invariants are functorial in homomorphisms $G\to G'$.   
   
 \begin{example}
  There are two extreme cases worth mentioning. 
  
  The first is $G=\GL_n$. In this case $\pi_1(G)=\BZ=\pi_1(G)_\Gamma$, and $\kappa_G([b])$ is determined by $\nu_{[b]}$. In  this case, the description of $B(G)$ by  $\nu_{[b]}$ coincides with the classical description of isocrystals by their Newton polygons (generalized to the case of $F$ instead of $\BQ_p$). 
  
  The second case is when $G=T$ is a torus. In this case, $\pi_1(T)=X_*(T)$ and $\kappa_G([b])$ determines $\nu_{[b]}$, which is simply the image of $\kappa_T([b])$ under the succession of maps
 \begin{equation*}
 X_*(T)_\Gamma\to  X_*(T)_\Gamma\otimes\BQ\simeq (X_*(T)_\BQ)^\Gamma .
 \end{equation*}
\end{example}
  
 \begin{remark}\label{Bexplicit}
  Let us make these invariants more explicit. 
 
  \smallskip
 
  (i) The Newton map is made explicit in \cite{Kottwitz1}, 4.3. The Kottwitz map for tori is made explicit in \cite{Kottwitz1}, 2.5., and the case when $G_{\rm der}$ is simply connected is reduced by functoriality to this case via the natural isomorphism $\pi_1(G)\simeq X_*(G_{\rm ab})$ that holds in this case. 
 
 \smallskip
 
 (ii) For unramified groups $G$, i.e., when $G$ is quasisplit and split over an unramified extension of $F$,  there is the following explicit description of $\kappa_G$.  Choose  a Borel subgroup $B$ of $G$ and a maximal torus $T$ contained in $B$, both defined over $F$. Then 
 $\pi_1(G)$ can be identified with the quotient of $X_*(T)$ by the coroot lattice, and the action of $\Gamma$ is the obvious one.
Let $b\in G(\breve F)$ and let $\mu\in X_*(T)$ be such that $b\in K\mu(\pi)K$ for some hyperspecial subgroup $K$ of $G(\breve F)$ that fixes a vertex in the apartment for $T_{\breve F}$. Here $\pi$ denotes a uniformizer of $F$. The existence of such a $\mu$ is a consequence of the Cartan decomposition. Then $\kappa_G([b])$ is the image of $\mu$ under the canonical projection from $X_*(T)$ to $\pi_1(G)_{\Gamma}$.  
 \smallskip

 (iii) Another way of making the invariants explicit is to relate $B(G)$ to the Iwahori-Weyl group, cf. \cite{He, Goertz-HN}. Recall the definition of the Iwahori-Weyl group, cf. \cite{HR}. Let $S$ be a maximal split $\breve F$-torus that is defined over $F$. Let $T$ be its centralizer. Since $G$ is quasisplit over $\breve F$, $T$ is a maximal torus. Let $N$ be its normalizer. The {\it finite Weyl group} associated to $S$ is
 \begin{equation*}
 W=N(\breve F)/T(\breve F) .
 \end{equation*}
 The {\it Iwahori-Weyl group} associated to $S$ is 
 \begin{equation*}
 \tilde W=N(\breve F)/T(\breve F)_1 ,
 \end{equation*}
 where $T(\breve F)_1$ denotes the unique Iwahori subgroup of $T(\breve F)$. Both groups are equipped with an action of $\sigma$.  Then $\tilde W$ contains the affine Weyl group $W_a$ as a normal subgroup, with factor group naturally identified with $\pi_1(G)_I$. Here $I=I_F=\Gal (\bar F/F^{{\rm un}})$ denotes the inertia group. The choice of a $\sigma$-invariant alcove in the apartment for $S$ splits the resulting  extension, and hence $\tilde W\simeq W_a\rtimes\pi_1(G)_I$.  The length function on $W_a$ is extended in a natural way to $\tilde W$ by regarding the elements of $\pi_1(G)_I$ as length zero elements. By further projection, we obtain a natural homomorphism
 \begin{equation}
 \kappa_{\tilde W}: \tilde W\to \pi_1(G)_\Gamma .
 \end{equation}
The element $w$ is called {\it straight} if for the  length  we have  
 $$
 \ell(w\sigma(w)\cdot\ldots\cdot\sigma^{n-1}(w))=n\ell(w), \,\forall n.
 $$ We denote the set of straight elements by $\tilde W_{\rm st}$. Furthermore, $\tilde W$ contains the abelian group $X_*(T)_I$ as normal subgroup, with factor group naturally identified with $W$. For $w\in \tilde W$, there exists an integer $n$ such that $\sigma^n$ acts trivially on $\tilde W$ and such that 
 $$
\lambda:= w\sigma(w)\cdot\ldots\cdot\sigma^{n-1}(w)\in X_*(T)_I .
 $$ The element  $\tilde\nu_w=\frac{1}{n}\lambda\in X_*(T)_\BQ^I$ is independent of $n$, and the conjugacy class of the corresponding homomorphism $\BD\to G_{\breve F}$ is defined over $F$. We denote this conjugacy class by $\nu_w$.   
 
 The explicit description of the invariants is now given as follows, cf.~\cite{He}. The map $N(\breve F)\to B(G)$ induces a {\it surjective} map 
 \begin{equation}\label{straight}
 \omega: \tilde W_{\rm st}\lra B(G) ,
 \end{equation}
 such that the following diagrams are commutative, 
 
 \begin{equation*}
 \begin{xy}
 \xymatrix@C-1em{
\tilde{W}_{\rm st} \ar[rr]^{\omega}\ar[rd]_{\kappa} &         & B(G) \ar[dl]^{\kappa}&&\tilde{W}_{\rm st} \ar[rr]^{\omega} \ar[rd]_{\nu} &         & B(G) \ar[dl]^{\nu} \\
				&   \pi_1(G)_{\Gamma}   &           & &		                         &    (X_*(G)_{\BQ}/G)^{\Gamma}   &
}
\end{xy}
\end{equation*}

 Furthermore, two elements in $\tilde W_{\rm st}$ map under $\omega$ to the same element in $B(G)$ if and only if they are $\sigma$-conjugate in $\tilde W$. 
 
 \smallskip
 
 (iv) There is also a cohomological description of the Kottwitz map which was communicated to us by L.~Fargues. It uses the cohomology theory of crossed modules \cite{Lab}, \S 1 and is inspired by the interpretation by Labesse of other cohomological constructions of Kottwitz. We consider the crossed module
 \begin{equation}
 [G_{\rm sc}(\breve F)\to G(\breve F)]\, ,
 \end{equation}
 where $G(\breve F)$ is in degree $0$ and $G_{\rm sc}(\breve F)$ is in degree $-1$. Here the structure of a crossed module is given by the conjugation action of $G(\breve F)$ on $G_{\rm sc}(\breve F)$. We form the cohomology of this crossed module wrt the action of $\sigma^\BZ$,
 \begin{equation}
 B_{\rm ab}(G):=H^1(\sigma^\BZ, [G_{\rm sc}(\breve F)\to G(\breve F)]) .
 \end{equation}
 We claim that $B_{\rm ab}(G)$ can be identified with $\pi_1(G)_\Gamma$. Indeed, let $T$ be a maximal torus of $G$ and let $T_{\rm sc}$ be the corresponding maximal torus of $G_{\rm sc}$. The homomorphism of crossed modules
 $$
 [T_{\rm sc}(\breve F)\to T(\breve F)]\to  [G_{\rm sc}(\breve F)\to G(\breve F)]
 $$
 is a quasi-isomorphism, hence 
 $$
 H^1(\sigma^\BZ, [T_{\rm sc}(\breve F)\to T(\breve F)]) \tilde{\lra} B_{\rm ab}(G) . 
 $$
 On the other hand, the LHS can be identified with 
 \begin{equation*}
 \begin{aligned}
 H^1(\sigma^\BZ, [T_{\rm sc}(\breve F)\to T(\breve F)]) \tilde{\lra} H_0(\Gamma, &[X_*(T_{\rm sc})\to X_*(T)]) \tilde{\lra}\\ &\tilde{\lra}H_0(\Gamma, X_*(T)/X_*(T_{\rm sc}))\tilde{\lra}\pi_1(G)_\Gamma  \, . 
 \end{aligned}
 \end{equation*}
 Here the first map  generalizes  Kottwitz's isomorphism 
 $$B(T)=H^1(\sigma^\BZ, T(\breve F))\tilde {\lra} X_*(T)_\Gamma ,
 $$ and the second isomorphism comes from the quasi-isomorphism 
 $$[X_*(T_{\rm sc})\to X_*(T)]\cong [1\to X_*(T)/X_*(T_{\rm sc})].
 $$ 
 Now the homomorphism of crossed modules 
 $$[1\to G(\breve F)]\to [G_{\rm sc}(\breve F)\to G(\breve F)],
 $$ which is equivariant for the action of $\sigma^\BZ$,  induces a natural map,
 \begin{equation}
 \kappa:B(G)\to B_{\rm ab}(G)\, .
 \end{equation}
 After identifying as above the target group with $\pi_1(G)_\Gamma$, this map can be identified with the Kottwitz map. This identification is compatible with the identification of $H^1(F, G)$ with $H^1_{\rm ab}(F, G)=H^1(\Gamma, G_{\rm sc}(\bar F)\to G(\bar F))$, the abelianized cohomology group  in \cite{Lab}, Prop. 1.6.7., in the sense of the commutativity of the following diagram, 
\begin{equation*}
\begin{xy}
 \xymatrix{
H^1 (F, G) \ar[r]^{\sim} \ar[d]   &         H^1_{ab} (F, G) \ar[d]   \\
B (G) \ar[r]^\kappa                  &         B_{ab} (G) \ .  \\
}
\end{xy}
\end{equation*}
Here the vertical arrows are the natural ones (cf. \cite{K}, 1.8 for the left vertical arrow).

 \end{remark}
Associated with $b\in G(\breve F)$ is the functor $J_b$ assigning to an $F$-algebra $R$ the group
$$J_b(R)=\{g\in G(R\otimes_F\breve F)\mid g(b\sigma)=(b\sigma)g)\} .$$
It is represented by an algebraic group $J_b$ over $F$. If $b'=g_0b\sigma(g_0^{-1})$, then conjugation  by  $g_0$ induces an isomorphism $J_b\cong J_{b'}$.  

A class $[b]$ is called {\it basic}, if the conjugacy class $\nu_{[b]}$ is central. In this case $J_b$ is an inner form of $G$. Denoting by $B(G)_{\rm basic}$ the set of basic elements in $B(G)$, there is a natural bijection
\begin{equation}
B(G)_{\rm basic}\simeq \pi_1(G)_\Gamma ,
\end{equation}
cf.~\cite{Kottwitz1}, Prop.~5.6. For a basic element $b\in G(\breve F)$, the cohomology class of the inner form $J_b$ is given by the image of the class $[b]$ of $b$ under the map 
$$
B(G)_{\rm basic}\to B(G_{\rm ad})_{\rm basic}\simeq H^1(F, G_{\rm ad}) ,
$$ 
cf. \cite{Kottwitz1}, Prop.~5.6.  In terms of the map \eqref{straight}, there is a bijection between the image  in $B(G)$ of the set of elements of $\tilde W$ of length zero and $B(G)_{\rm basic}$ (the elements of length zero are all straight). 

\smallskip

%Let us illustrate how this theory is related to the classical description of $F$-isocrystals by Newton polygons. We choose $G=GL_n$. Let $B$ be the group of lower triangular matrices and let $T$ be the diagonal torus. We have $X_*(T)\cong \mathbb{Z}^n$ with trivial $\Gamma$-action. A (dominant) Newton point of an element $b\in G(L)$ thus corresponds to an element $\nu_b=(\nu_1,\dotsc,\nu_n)\in \mathbb{Q}^n$ with $\nu_1\leq\dotsm\leq\nu_n$. With $\nu_b$ one also associates the polygon that is the graph of the continuous, piecewise linear function $[0,n]\rightarrow \mathbb{R}$ having slope $\nu_i$ on $[i-1,i]$. This establishes the analogy with the usual notion of the Newton polygon of an $F$-isocrystal, in this case of the $n$-dimensional $F$-isocrystal given by $F=b\sigma$. An ordered element of $\mathbb Q^n$ arises as a Newton point if and only if all break points of the associated polygon as well as its end point have integral coordinates. For $G=GL_n$, the Newton point $\nu_b$ then determines $\kappa_G(b)\in\pi_1(G)\cong \mathbb{Z}$ by $\kappa_G(b)=\sum_{i=1}^n \nu_i$.

%Let us return to the general context. 

To explain non-basic classes,  let us assume for simplicity that $G$ is quasisplit. Let $M$ be a Levi subgroup of $G$, i.e., a Levi component of a parabolic subgroup defined over $F$. A basic element $b\in M(\breve F)$ is called {\it $G$-regular} if the centralizer in $G$ of the homomorphism $\tilde \nu_b$  is equal to $M$, cf. \cite{Kottwitz1}, \S 6. Then any $[b]\in B(G)$ is the class of a $G$-regular basic element in a Levi subgroup, and this in an essentially unique way, cf. \cite{Kottwitz1}, Prop. 6.3. 

More explicitly, this can be described as follows. Let $B$ be a Borel subgroup and $T$ a maximal torus contained in $B$, both defined over $F$. 
 Let  $X_*(T)_{\dom}$, resp. $\big(X_*(T)_\BQ\big)_{\dom}$ be the set of cocharacters of $T$, resp. elements of $X_*(T)_\BQ$, which are dominant with respect to $B$. Then $\Gamma$ acts on these sets. Each conjugacy class of homomorphism $\tilde\nu:\mathbb{D}\rightarrow G_{\breve F}$ has a unique dominant representative in $ \big(X_*(T)_\BQ\big)_{\dom}$, which is $\Gamma$-invariant if the conjugacy class is $\Gamma$-invariant. For a given $[b]\in B(G)$, we often use the same notation $\nu_{[b]}$   for this representative of the conjugacy class $\nu_{[b]}$ defined by $[b]$. For $[b]\in B(G)$ we denote by $M_{[b]}$ the centralizer of $\nu_{[b]}$, a standard Levi subgroup.   The class $[b]$ then has a representative $b$ in $M(\breve F)$, and yields a well-defined class in $B(M)_{\rm basic}$. 
The group $J_b$,  for $b\in [b]$, is an inner form of  $M_{[b]}$, cf. \cite{Kottwitz1}, Remark 6.5.

\subsection{Acceptable elements}
We continue to assume that $G$ is quasisplit. For  a dominant element $\mu\in X_*(T)$,  let 
\begin{equation}
\overline{\mu}=[\Gamma:\Gamma_{\mu}]^{-1}\sum_{\tau\in\Gamma/\Gamma_{\mu}}\tau(\mu)\in(X_*(T)_{\mathbb{Q}})_{\dom}^{\Gamma}, 
\end{equation}
 where we denote by $\Gamma_{\mu}$ the stabilizer of $\mu$ in $\Gamma$. 
 On $X_*(T)_{\mathbb{Q}}$ we consider the order $\leq$ given by $\nu\leq\nu'$ if and only if $\nu'-\nu$ is a non-negative $\BQ$-linear combination of positive {\it relative} coroots, cf.~\cite{DOR}, ch. VI, \S 6. 
\begin{definition}\label{defadm} Let $\{\mu\}$ be a conjugacy class of cocharacters over $\bar F$ of $G$. An element $[b]\in B(G)$ is called {\it acceptable} for $\{\mu\}$ if   $\nu_{[b]}\leq \overline\mu_{\dom}$, where  
 $\mu_{\dom}$ denotes the unique element of $\{\mu\}$  contained in $X_*(T)_{\dom}$. An acceptable element 
 $[b]$ is called {\it neutral} if  $\kappa_G([b])=\mu^\natural$, where $\mu^\natural$ denotes the image of any $\mu\in \{\mu\}$ in 
$\pi_1(G)_{\Gamma}$. 

Let $A(G,\{\mu\})$ be the set of acceptable elements for $\{\mu\}$, and $B(G,\{\mu\})$  the subset of $A(G,\{\mu\})$ of neutral acceptable elements. 
\end{definition}
The sets $A(G,\{\mu\})$ and $B(G,\{\mu\})$ only depend on the image of $\mu_{\dom}$ in $X_*(T)_\Gamma$, cf. \cite{K2}, 6.3. 
\begin{remark} We will use these sets also when $G$ is no longer quasisplit. For  the definition of $B(G, \{\mu\})$ in the general case we refer to  \cite{K2}, 6.2, and the   set $A(G, \{\mu\})$ is defined analogously. The set $B(G, \{\mu\})$ appears in the description of points modulo $p$ of Shimura varieties, \cite{Rapo-guide}, \S 9, and the set $A(G,\{\mu\})$ appears in the context of period spaces, compare Proposition \ref{propperne} below. For the group $G=GL_n$ both sets coincide, and in this case the inequality defining them is nothing but a group-theoretical reformulation of the Mazur inequality for $F$-crystals, cf. \cite{RapoportRichartz, K}. 
\end{remark}

\begin{lemma} 
The sets $A(G, \{\mu\})$ and $B(G, \{\mu\})$ are  finite and non-empty. 
\end{lemma}
\begin{proof} Non-emptiness follows because $B(G,\{\mu\})$ contains the unique basic element of $B(G)$ with $\kappa_G([b])=\mu^\natural$, cf. \cite{K2}, 6.4. Finiteness of $A(G,\mu)$ is a special case of \cite{RapoportRichartz}, Prop. 2.4 (iii). In particular, the subset $B(G, \{\mu\})$ is also finite. 
\end{proof} 
\begin{remark} Here is another argument to prove the finiteness of $A(G, \{\mu\})$, based on the interpretation of $B(G)$ in terms of  the Iwahori-Weyl group, cf. (iii) of Remark  \ref{Bexplicit} (we learned this argument from X.~He). If $w\in \tilde W_{\rm st}$ represents an element in $A(G, \{\mu\})$, then the length of $w$ is bounded by the length of the translation element to the image of $\mu$ in $X_*(T)_I$, whereas its image in $\pi_1(G)_\Gamma\otimes \BQ$ is fixed. But there are only finitely many such elements. 

\end{remark}
We have an identification
$H^1(F,G)=(\pi_1(G)_{\Gamma})_{{\rm tors}}$, cf. \cite{Kottwitz1}, Remark~5.7.
Hence, since the images of $\kappa_G([b])$ and $\mu^\natural$ in $\pi_1(G)_\Gamma\otimes\BQ$ are identical,   we obtain  a map 
 \begin{equation}
 c:A(G, \{\mu\})\rightarrow H^1(F,G)\quad ([b],\mu)\mapsto c([b], \{\mu\}):=\kappa_G([b])-\mu^\natural .
 \end{equation}
  By definition, the inverse image of $0$ under this map coincides with the set $B(G,\mu)$.  For $[b]\in A(G, \{\mu\})$, we denote by $G_{[b]}$ the inner form of $G$ corresponding to the image in $G_{\rm ad}$ of any cocycle in the class $c([b],\mu)$. 
 
Recall that a pair $(b, \mu)$ consisting of an element $b\in G(\breve F)$ and a cocharacter $\mu$ of $G$  defined over a finite extension $K$ of $\breve F$ is called a {\it weakly  admissible pair} in $G$ if for any $F$-rational representation $V$ of $G$, the filtered $\sigma$-$F$-space $\CI(V)=(V\otimes_F\breve F, b\sigma, Fil^\bullet_\mu)$ is weakly admissible, cf.~\cite {DOR}, p.~219. Here weak admissibility of a filtered $\sigma$-$F$-space is the obvious generalization of Fontaine's original definition, e.g., \cite{DOR}, Def.~8.1.6.  It is enough to check this property on one faithful representation of $G$, cf. loc. cit. Also, this notion only depends on the {\it par-equivalence} class of $\mu$, i.e., only on the filtration on the category ${\rm Rep}_F(G)$ of 
finite-dimensional $F$-rational representations of $G$ induced by $\mu$, cf. \cite{DOR}, Cor.~9.2.26. There is an equivalence of categories induced by Fontaine's functor
\begin{equation*}
\begin{aligned}
\omega: &\{\text{weakly admissible  filtered $\sigma$-$F$-spaces}\}\to \\ &\quad\quad\quad\quad\quad\quad\quad\quad\quad\{ \text{crystalline \emph{special} $F$-rep's of $\Gal(\bar{\breve {F}}/K)$}\} . 
\end{aligned}
\end{equation*}

Here  an $F$-representation is a continuous representation of $\Gal(\bar{\breve {F}}/K)$ in a finite-dimensional $F$-vector space; it is called special if it satisfies the {\it special} condition $(*)$ of \cite{FR}, Introd. (If the $F$-representation comes from a $p$-divisible formal group $X$, the condition $(*)$ is equivalent to the condition that $X$ is a formal $\CO_F$-module in the sense of Drinfeld, cf. \cite{Dr}, \S 1.) Such a representation is called crystalline if its underlying $p$-adic representation is crystalline.

For a weakly admissible pair, consider the following two fiber functors on ${\rm Rep}_F(G)$. The first is the standard fiber functor  $Ver$. The other functor, denoted by  $\tilde{\mathcal{F}}$, maps $V\in {\rm Rep}_F(G)$  to the underlying $F$-vector  space of the crystalline special 
$F$-representation of $\Gal(\bar{\breve {F}}/K)$ associated  by the  functor $\omega$ to the weakly admissible filtered $F$-isocrystal $\CI(V)$. Then $\Hom (Ver,\tilde{\mathcal{F}})$ is a right $G$-torsor, hence  defines an element $\cls(b,\mu)\in H^1( F,G).$ 
\begin{proposition}\label{Fontaininequ}
Let $(b, \mu)$ be a weakly admissible pair in $G$. Then the $\sigma$-conjugacy class $[b]$ of $b$ is acceptable for $ \{\mu\}$.  There is an equality of classes in $H^1(F,G)$,
$$
\cls(b,\mu)=c([b], \{\mu\}).
$$
\end{proposition}
\begin{proof} See \cite{RZ}, Prop. 1.20 in the case that $G_{\der}$ is simply connected. The general case is due to  Wintenberger  \cite{W}.
\end{proof}

\section{Period domains}\label{perdom}

Let $G$ be a reductive algebraic group over $F$. Period spaces arise when fixing an element $b\in G(\breve F)$ and a conjugacy class $\{\mu\}$ of cocharacters of $G$ over $\bar F$. Let $E=E_{\{\mu\}}$ be the field of definition of $\{\mu\}$, the {\it reflex field} of the PD-pair $(G, \{\mu\})$, cf. \cite {DOR}, Definition 6.1.2. Then $E$ is a finite extension of $F$ contained in the fixed algebraic closure $\bar F$ of $F$. 

There is a smooth projective variety $\CF=\CF(G, \{\mu\})$ over $E$, whose points over $\bar F$ correspond to the par-equivalence classes of elements in 
$\{\mu\}$, cf. \cite {DOR}, Theorem 6.1.4.  The variety $\CF$ is homogeneous under $G_E$, and is a generalized flag variety for $G_E$.  

Let $\breve E=\breve F.E$, where the composite is taken inside an algebraic closure of $\breve F$ containing $\bar F$. Then $\breve E$ is the completion of the maximal unramified extension of $E$ in $\bar F$. Let 
\begin{equation}
{\breve {\mathcal F}(G, \{\mu\})}^{\rig}=\breve {\mathcal F}^{\rig}=\big(\mathcal{F}\times_{\Spec\, E}\Spec\, \breve E\big)^\rig
\end{equation}
denote the 
rigid-analytic space associated with the projective variety  $\breve{\mathcal F}=\mathcal{F}\times_{\Spec\, E}\Spec\, \breve E$ over the discretely valued field $\breve E$. The set of  elements $\mu\in{\breve\CF}(\bar F)$ such that $(b, \mu)$ is a weakly admissible pair for $G$ forms an admissible open subset of ${\breve\CF}^\rig$, called the {\it non-archimedean period domain} associated with $(G,b,\{\mu\})$ and denoted by $\breve{\mathcal{F}}(G,b,\{\mu\})^{{\rm wa}}$, cf. \cite{RZ}, 1.35. The period domain is stable under the action of $J_b(F)$ on $\breve{\mathcal F}^\rig$. 

In the sequel, we will sometimes consider $\breve{\mathcal F}^\rig$ as an adic space. We also have use for the version of  $\mathcal{F}(G,b,\{\mu\})^{{\rm wa}}$ in the category of Berkovich spaces, cf. \cite {DOR}, Def.~9.5.4 (see \cite{Fargues}, App. D and  \cite{Scholze-perf}, \S 2 for  succint explanations of the relation between rigid-analytic, adic and Berkovich spaces).

The period domain $\breve{\mathcal{F}}(G, b,\{\mu\})^{{\rm wa}}$  depends up to isomorphism only on the $\sigma$-conjugacy class $[b]$ of $b$. Indeed, if $b'=gb\sigma(g)^{-1}$, then the translation map $\mu\mapsto \mu'={\rm Ad}(g)(\mu)$ induces such an isomorphism $\breve{\mathcal{F}}(G, b,\{\mu\})^{{\rm wa}}\tilde{\lra}\breve{\mathcal{F}}(G, b',\{\mu\})^{{\rm wa}}$. 

If $\mathcal{F}(G,b,\{\mu\})^{{\rm wa}}$ is non-empty, then $[b]\in A(G, \{\mu\})$, cf. Proposition \ref{Fontaininequ}. The converse is also true, cf. \cite{DOR}, Thm.~9.5.10: 
\begin{proposition}\label{propperne} 
Let $b\in G(\breve F)$, and let $\{\mu\}$ be a conjugacy class of cocharacters of $G_{\bar F}$. Then $[b]\in A(G, \{\mu\})$ if and only if $\breve{\mathcal{F}}(G, b,\{\mu\})^{{\rm wa}}$ is non-empty. \qed
\end{proposition}

The best-known example of a period domain is the Drinfeld half space $\Omega^n_F$, attached to the triple $G=\GL_n$, $ [b]=[1]$, $\{\mu\}=\big(n-1, (-1)^{(n-1)}\big)$. It is the complement in $\BP^{n-1}_F$ of all $F$-rational hyperplanes. The same rigid-analytic variety arises also from other triples $(G, [b], \{\mu\})$. We refer to \cite{DOR}, Chapters VIII, IX,  for an exposition of the theory of $p$-adic period domains, with many more examples.

\section{RZ spaces}\label{rzspaces}

In this section we recall the formal moduli spaces of $p$-divisible groups within a fixed isogeny class constructed in \cite{RZ}. These have become known under the name of {\it Rapoport-Zink spaces}, cf., e.g., \cite{Fargues}. In this paper, we abbreviate this to {\it RZ spaces}.

\subsection{RZ data}\label{RZdata}

Let us review the data on which these moduli spaces depend. We adapt the original definition \cite{RZ}, Definition 3.18,  to the case we are interested in. We distinguish two classes, the EL case and the PEL case.

\smallskip

\noindent{\bf EL case.} A {\it simple rational} RZ {\it datum} in the EL case is a tuple $\mathcal{D}$ of the form $\CD=(F,B,V, \{\mu\},[b])$ where
\begin{altitemize}
\item $F$ is a finite field extension of $\mathbb{Q}_p$
\item $B$ is a central division algebra over $F$
\item $V$ is a finite-dimensional $B$-module.
\item $\{\mu\}$ is a conjugacy class of {\it minuscule} cocharacters $\mu:\mathbb{G}_{m,\overline{\mathbb{Q}}_p}\rightarrow G_{\overline {\mathbb Q}_p}$ 
\item $[b]\in A(G,\{\mu\})$. 
\end{altitemize}
Here 
$$G=\GL_B(V), $$
  viewed as an algebraic group over $\mathbb{Q}_p$. 
    
  We impose the following condition. For $\mu\in\{\mu\}$ consider  the  decomposition of $V\otimes\bar\BQ_p$ into weight spaces. We then require  that only weights $0$ and $1$ occur in the decomposition, i.e. $V\otimes\bar\BQ_p = V_0 \oplus V_1$. 
  
A {\it simple integral} RZ {\it datum} $\CD_{\BZ_p}$ in the EL case consists, in addition to the data $\CD$,  of a maximal order $\CO_B$ in $B$ and an $\CO_B$-stable lattice $\Lambda$ in $V$. This induces an integral model $\CG$ of $G$ over $\BZ_p$, namely $\CG=\GL_{\CO_B}(\Lambda)$, considered as a group scheme over $\BZ_p$. 

\begin{remark}
 In \cite{RZ} more general (not necessarily simple) integral RZ data are considered. (Instead of a single lattice $\Lambda$,  {\it periodic lattice chains} are considered.)   In the sequel we will sometimes refer to those as well. The group scheme $\CG$ has connected special fiber. The subgroup $\CG(\BZ_p)$ of $G(\BQ_p)$ is a parahoric subgroup, and in fact a maximal parahoric subgroup if $B=F$. To obtain in the case $B=F$ more general parahoric subgroups (e.g., Iwahori subgroups), one has to consider more general periodic lattice chains. 
\end{remark}

\noindent{\bf PEL case.} This case will only be considered when $p\neq 2$. A {\it simple rational} RZ {\it datum} in the PEL case is a tuple $\mathcal{D}=(F,B,V, (\, , \, ), *,  \{\mu\},[b])$ where
  
\begin{altitemize}
\item $F$, $B$, and $V$ are as in the EL case
\item $(\, ,\, )$ is a non-degenerate alternating $\mathbb{Q}_p$-bilinear form on $V$
\item $*$ is an involution on $B$ satisfying $$(xv,w)=(v,x^*w) ,\,\,  \forall v,w\in V, \, \forall x\in B .$$
\item $\{\mu\}$ is a conjugacy class of minuscule cocharacters $\mu:\mathbb{G}_{m,\overline{\mathbb{Q}}_p}\rightarrow G_{\overline {\mathbb Q}_p}$,  where $G$ is the algebraic group over $\BQ_p$ defined by
\begin{multline*}
G(R)=\{g\in \GL_{B\otimes_{\mathbb{Q}_p}R}(V\otimes_{\mathbb{Q}_p}R)\mid \exists c(g)\in R^{\times}:\\ (gv_1,gv_2)=c(g)(v_1,v_2),\,  \forall v_1,v_2\in V\otimes_{\mathbb{Q}_p}R \}.
\end{multline*}
\item $[b]\in A(G,\{\mu\})$
\end{altitemize}
We impose the following properties. First, we require, as in the EL case,  that, for any $\mu\in\{\mu\}$,  the only weights of $\mu$ occurring in $V\otimes\bar\BQ_p$ are $0$ and $1$. Secondly, we require that for any $\mu\in\{\mu\}$ the composition
\begin{equation}\label{eqcmu}
\BG_{m, \bar \BQ_p}  \buildrel{\mu}\over\longrightarrow G_{\bar \BQ_p} \buildrel{c}\over\longrightarrow \BG_{m, \bar \BQ_p}
\end{equation}
 is the identity, the latter morphism denoting the multiplier $c: G\to \BG_m$.
 
 A {\it simple integral}  RZ {\it datum} $\CD_{\BZ_p}$ in the PEL case consists, in addition to the data $\CD$,  of a maximal order $\CO_B$ in $B$ which is stable under the involution $*$, and an $\CO_B$-stable lattice $\Lambda$ in $V$ such that $\varpi\Lambda\subset \Lambda^\vee\subset \Lambda$. Here $\varpi$ denotes a uniformizer in $\CO_B$, and $\Lambda^\vee$ denotes the lattice of elements in $V$ which pair integrally with $\Lambda$ under the symplectic form $(\, , \, )$. This induces an integral model $\CG$ of $G$ over $\BZ_p$, with $\CG(\BZ_p)=G(\BQ_p)\cap \GL_{\CO_B}(\Lambda)$. 
\begin{remark}
Again, in \cite{RZ} more general integral data are considered (these correspond to {\it selfdual periodic lattice chains}). The subgroup $\CG(\BZ_p)$ is not always a parahoric subgroup of $G(\BQ_p)$; however, this holds if $B=F$ and the involution $*$ is trivial, or if $B=F$ and the extension $F/F_0$ is an {\it unramified} quadratic extension. Here $F_0$ denotes the fix field of $*$ in $F$. Furthermore, in these cases, the parahoric subgroup in question is maximal. 

On the other hand, consider the case when $B=F$ and $F/F_0$ is a ramified quadratic extension, and when $V$ is even-dimensional. Then the stabilizer of $\Lambda$ is {\it never} a parahoric subgroup unless $\Lambda^\vee=\varpi\Lambda$, cf. \cite{PRLM3}, 1.b.3, b). More precisely, in these cases the group $\CG(\BZ_p)$ is not parahoric, but contains a unique parahoric subgroup with index $2$, cf. \cite{PRLM3}, loc.~cit. The reason for this phenomenon is that the group scheme $\CG$ has in this case a special fiber that has two connected components, whereas parahoric group schemes are by definition connected. 
\end{remark}
\begin{remark}
 In the PEL case the group scheme $G$ above is not always connected. When it is not connected, $G$ does not fit into the framework of section \ref{sigmaconjclas}. In these cases (which are all related to  orthogonal groups), one has to use ad hoc definitions for $A(G, \{\mu\})$ etc., cf. \cite{smithling}. In the sequel we will neglect this difficulty, especially since the putative theory of local Shimura varieties should take as its point of departure  purely group-theoretic data.   However, it seems reasonable to expect that suitable modifications of the theory outlined here should transfer to the cases related to  orthogonal groups as well. 
\end{remark}
\begin{remark}\label{compwithRZ} The conditions on RZ data in \cite{RZ} seem to differ from those imposed here in one point. In \cite{RZ}, instead of the condition that $[b]\in A(G, \{\mu\})$ that appears here,  one asks  for the existence of a weakly admissible pair $(b, \mu)$ with $b\in [b]$ and $\mu\in \{\mu\}$. It follows from Proposition \ref{propperne} that the two conditions are in fact equivalent. These conditions  are satisfied   when the rigid space associated to the formal scheme in subsection \ref{secrz} is non-empty. The converse is discussed in Question \ref{questnonempty} and in Conjecture \ref{conjnonempty} below. 
\end{remark}
In both the EL case and the PEL case, we denote by  $E=E_{\{\mu\}}$  the associated reflex field, i.e., the  field of definition of $\{\mu\}$. Also, for fixed $b\in [b]$, we will denote by $J$  the algebraic group $J_b$ of section \ref{sigmaconjclas}. 
We write $\mathcal{O}_E$, resp.~$\CO=\mathcal{O}_{\breve E}$, for the ring of integers of $E$,  resp.~of $\breve E$.

 \subsection{The formal schemes} \label{secrz}
 
  We explain the formal schemes over $\Spf\, \CO$ associated to integral RZ data $\CD_{\BZ_p}$ of EL type or PEL type. 
  
  Let $\mathcal Nilp_{\mathcal O}$ denote the category of $\mathcal O$-schemes $S$ on which $p$ is locally nilpotent. For $S\in\mathcal Nilp_{\mathcal O}$ we denote by $\overline{S}$ the closed subscheme defined by $p\mathcal{O}_S$.
  
  Let us first consider the EL case. In what follows we will consider  pairs $(X, \iota)$, where $X$ is a $p$-divisible group over $S\in \mathcal Nilp_{\mathcal O}$ and $\iota:\CO_B\to\End\, X$ is an action of $\CO_B$ on $X$. We require the {\it Kottwitz condition} associated to $\{\mu\}$, i.e., the equality of characteristic polynomials, 
  \begin{equation}\label{kottwitzbed}
  {\rm char}\,(\iota(b);\Lie\, X)={\rm char}\,(b; V_0) , \quad \forall b\in\CO_B,
  \end{equation}
  where $V_0$ is determined by the weight decomposition of $V\otimes\bar\BQ_p$ associated to any $\mu\in\{\mu\}$. The isomorphism class of the $B$-module $V_0$ is defined over $E$ and independent of the choice of $\mu$. Since $\CO_B$ preserves $\Lambda$,  the RHS of \eqref{kottwitzbed} is a polynomial with coefficients in $\CO_E$; it  is compared with the polynomial with coefficients in $\CO_S$ on the LHS via the structure morphism. For a comparison with the original version of this condition \cite{K}  (which is more abstract), see \cite{Hartwig}, Prop.~2.1.3. 
  
Let $\BF$ be the residue field of $\CO$, and fix a pair $(\BX, \iota_\BX)$ as above (a {\it framing object}). We demand that the rational Dieudonn\'e module of $\BX$ with its action by $B$ and its  Frobenius endomorphism is isomorphic to $(V\otimes_{\BQ_p}\breve\BQ_p, b\sigma)$, where $b\in [b]$ is a fixed element. Then we consider the set-valued functor $\CM_{\CD_{\BZ_p}}$ on $\mathcal Nilp_\CO$ which associates to $S\in\mathcal Nilp_\CO$ the set of isomorphism classes of triples $(X, \iota, \rho)$, where $(X, \iota)$ is as above, and where
\begin{equation*}
\rho: X\times_S\bar S\to \BX\times_{\Spec\, \BF}\bar S
\end{equation*}
is an $\CO_B$-linear quasi-isogeny. (Note that in \cite{RZ} the quasi-isogeny is in the other direction; this entails a few changes of a trivial nature.)

Now let us consider the PEL case. In what follows we will consider triples $(X, \iota, \lambda)$ over schemes $S\in\mathcal Nilp_{\mathcal O}$, where $(X, \iota)$ is as in the EL case, and where $\lambda$ is a polarization of $X$ whose associated Rosati involution induces on $\CO_B$ the given involution $*$, and such that $\Ker\lambda\subset X[\iota(\varpi)]$ has order $|\Lambda/\Lambda^\vee|$. We fix a framing object $(\BX, \iota_\BX, \lambda_\BX)$ such that its rational Dieudonn\'e module with its action by $B$ and its polarization form and its Frobenius endomorphism is isomorphic to $(V\otimes_{\BQ_p}\breve\BQ_p, (\, ,\, ), b\sigma)$  (the isomorphism preserving the forms up to a factor in $\breve\BQ_p^\times$). Then we consider the set-valued functor $\CM_{\CD_{\BZ_p}}$ on $\mathcal Nilp_\CO$ which associates to $S\in\mathcal Nilp_\CO$ the set of isomorphism classes of triples $(X, \iota, \lambda, \rho)$, where $(X, \iota, \lambda)$ is as above, and where
\begin{equation*}
\rho: X\times_S\bar S\to \BX\times_{\Spec\, \BF}\bar S
\end{equation*}
is an $\CO_B$-linear quasi-isogeny that preserves the polarizations up to a factor in $\BQ_p^\times$, locally on $\bar S$. 
\begin{theorem}
Let $\CD_{\BZ_p}$ be integral RZ data of type EL or PEL. The functor $\CM_{\CD_{\BZ_p}}$ on $\mathcal Nilp_{\mathcal O}$ is representable by a formal scheme,  locally formally of finite type and separated over $\Spf\, \CO$.
\end{theorem}
\begin{proof}
See \cite{RZ}, Cor.~3.40.  For the separatedness assertion, see \cite{Fargues}, Lemma 2.3.23.
\end{proof}
We note that the group $J_b(\BQ_p)$ acts on $\CM_{\CD_{\BZ_p}}$ by post composing with $\rho$,
\begin{equation}
g: (X, \iota, \lambda)\mapsto (X, \iota, g\circ \rho) , \quad g\in J_b(\BQ_p) . 
\end{equation} 
This action is continuous, cf. \cite{Fargues}, Prop.~2.3.11 and Rem.~2.3.12. 

Note that $\CM_{\CD_{\BZ_p}}$ only depends on the chosen $p$-divisible group $\mathbb{X}$ up to isogeny ($\CO_B$-linear isogeny in the EL case, resp. $\CO_B$-linear isogeny that preserves the polarization up to a scale factor in $\BQ_p^\times$ in the PEL case). Up to isomorphism $\CM_{\CD_{\BZ_p}}$ is independent of the choice of $b\in [b]$. 

\begin{remark} \label{remonflatdef} It is not clear whether a framing object always exists, but this seems quite possible. Related to this question is what condition one wants to impose on the action of $\CO_B$ on $\Lie\, X$. We have imposed only the Kottwitz condition but, especially in more ramified cases, one might want to impose other conditions, like the wedge condition or the spin condition, cf. \cite{PRLM3}. Then, of course, it becomes more difficult to find a framing object satisfying these conditions as well, cf. Example \ref{exempty} below. Generally speaking, one would like to impose additional conditions such that one obtains a closed formal subscheme of  $\CM_{\CD_{\BZ_p}}$ which is flat over $\CO$ and with underlying reduced scheme equal to the {\it affine Deligne-Lusztig variety} associated to $(G, b, \{\mu\}, \CG)$, cf.~\cite{RapoportRichartz}. The question of when an affine Deligne-Lusztig variety is non-empty has been much studied but is still open in the most general case, cf.~\cite{Gashi,GHKR,Goertz-HN,He,RapoportRichartz}.

In any case, having fixed the element $b\in [b]$, simple rational RZ data define the rational Dieudonn\'e module of the searched for framing object. By Dieudonn\'e theory, the existence of a framing object then comes down to finding a Dieudonn\'e lattice inside this rational Dieudonn\'e module with the required properties (stability under the action of $\CO_B$, almost self-duality, action on the tangent space, etc.). 
\end{remark}
\begin{remark}\label{natdef}
From the point of view of the formulation of the moduli problem, starting with the RZ datum $\CD_{\BZ_p}$ is not very natural. Indeed, from this point of view it is more natural to fix the framing object $\BX$ (with its additional structures) and to use it to formulate a moduli problem $\tilde\CM$ as above. Then any integral RZ datum $\CD_{\BZ_p}=(F, B, V, (\, , \,), *, \{\mu\}, [b], \CO_B, \Lambda)$ with an identification of the Dieudonn\'e module of $\BX$ respecting all additional structures and carrying the Frobenius endomorphism into $b\sigma$,
$$
M(\BX)\simeq\Lambda\otimes_{\BZ_p} W(\BF) ,
$$ 
defines an isomorphism of formal $\CO$-schemes $\CM_{\CD_{\BZ_p}}\simeq \tilde\CM$. 

Such an integral RZ datum $\CD_{\BZ_p}$ exists. Suppose a second choice $\CD'_{\BZ_p}$ is possible, and let $G$, resp. $G'$, be the two algebraic groups over $\BQ_p$ associated to the simple rational RZ data $\CD$, resp. $\CD'$. Then the two choices $\CD$ and $\CD'$ are related by a cohomology class $\eta\in H^1(\BQ_p, G)$. Then $G'$ is an inner form of $G$ (corresponding to the image of $\eta$ in $ H^1(\BQ_p, G_{\rm ad})$) and the two groups $\pi_1(G')_\Gamma$ and $\pi_1(G)_\Gamma$ are canonically identified. Under this identification, there is the identity
\begin{equation}\label{twistcomp}
c([b], \{\mu\})=c([b'], \{\mu'\})-\eta ,
\end{equation}
cf.~\cite{RZ}, Lemma 1.27. Here $\eta$ is considered as an element in $(\pi_1(G)_{\Gamma})_{\rm tors}$. 

This seems to point to the possibility of using {\it augmented group schemes} over the Tannaka category of isocrystals, in the sense of \cite{DOR}, Def.~9.1.15, as basic input in the definition of RZ data. However, we decided to stay with the present set-up since it seems more elementary. 
\end{remark}
\begin{remark}  Let ${\CD_{\BZ_p}}$ be an integral RZ datum, and $\CM=\CM_{\CD_{\BZ_p}}$ the associated formal scheme. Let $\CM_{\rm red}$ denote the underlying reduced scheme of $\CM$. This is a scheme locally of finite type over $\BF$, such that all its irreducible components are projective varieties over $\BF$, cf. \cite{RZ}, Prop.~2.32. It is an interesting problem to give an explicit description of this scheme. This has been done in a number of  cases, especially in relation to the uniformization theorem \cite{RZ}, ch.~6,  which shows that this problem is closely related to the problem of giving an explicit description of the {\it basic locus} in the reduction modulo $p$ of Shimura varieties. Let us discuss some examples;  for a more complete discussion and references compare the introduction of \cite{V}. All examples will correspond to basic elements in $B(G, \{\mu\})$. 

Let us consider the PEL case where $F=B=\BQ_p$, in which case $G={\rm GSp}(V, (\, ,\,))$, where $V$ is a symplectic vector space of dimension $2g$. We assume that $\Lambda$ is a selfdual lattice in $V$.   For $g=1$, 
$\CM_{\rm red}$ is a discrete sets of points. The case $g=2$ has been described by Koblitz \cite{Kob}, Katsura and Oort \cite{KO1}, Kaiser \cite{Ka} and Kudla and Rapoport \cite{KR2}. In this case the set of irreducible components of  $\CM_{\rm red}$ and their intersection behaviour is closely related to the Bruhat-Tits building of the associated group $J$. The irreducible components are all  projective lines. For $g=3$ partial results are obtained by Katsura and Oort
\cite{KO2}, Li and Oort \cite{LO} and Richartz \cite{Ri}. In this case again, the set of irreducible components can be described in terms of the Bruhat-Tits building of $J$. However, the individual irreducible components have a complicated geometric structure.

Another class of PEL examples arises when $F=B$ and where $*$ is non-trivial. In this case $G={\rm GU}(V,(\, ,\,))$ is the group of unitary similitudes, where $(V,(\, ,\,))$ is a hermitian vector space over $F$. Again, we take $\Lambda$ to be self-dual. We assume that the weight spaces for $\mu$  are of dimension $1$ and $n-1$, where $n=\dim V\geq 2$. In this case, one can again describe the set of irreducible components in terms of the Bruhat-Tits building of $J$; the individual irreducible components can be identified with compactified Deligne-Lusztig varieties (Vollaard and Wedhorn \cite{VollaardWedhorn} for the case when $F/F_0$ is unramified; Rapoport, Terstiege and Wilson \cite{RTW} for the case when $F/F_0$ is ramified---note that in the latter case, which is ramified, additional conditions are imposed to define a flat formal scheme over $\CO$, cf. Remark \ref{remonflatdef}). Very recently, the case where $n=4$ and where the weight spaces for $\mu$ are both of dimension $2$, and where the extension $F/F_0$ is unramified has been cleared up by Howard and Pappas \cite{HP}. 

Other examples that have been studied are of Hilbert-Blumenthal type (Bachmat and Goren \cite{BG},  Goren  \cite{Gn},  Yu \cite{Yu}, Goren and Oort \cite{GO}, Stamm \cite{St}, Kudla and Rapoport \cite{KR1}).

G\"ortz and He \cite{GH} study a closely related question in the equal characteristic case. Instead of the reduced locus of  formal schemes of the form $(\CM_{\CD_{\BZ_p}})_{\rm red}$, they  consider certain affine Deligne-Lusztig varieties  which are defined as subschemes of affine Grassmannians, cf.~\cite{GHKR}. In \cite{GH} the question is studied of when the set of irreducible components  can be described in terms of the Bruhat-Tits building of $J(F)$, and when the individual irreducible components are isomorphic to  Deligne-Lusztig varieties. As the examples related to symplectic groups of rank $\geq 3$ show, this is possible only in a limited number of cases, and G\"ortz and He give an essentially complete list of them. \end{remark}

\begin{remark}\label{locmodrem} The local structure of the formal scheme $\CM=\CM_{\CD_{\BZ_p}}$ is given by the associated local model (in fact, the {\it naive } local model since we only impose the Kottwitz condition on the Lie algebra). The naive local model is a projective scheme $M^{\rm naive}$ over $\Spec\, \CO_E$ such that, for every geometric point in $\CM_{\rm red}$, there exists a geometric point in the special fiber of $M^{\rm naive}$ such that the complete local rings of $\CM$ and of $M^{\rm naive}$ at these points are isomorphic. In fact, there is a {\it local model diagram} connecting $\CM$ and $M^{\rm naive}$, comp. \cite{PRS}, \S 1.2.  Naive local models have been mostly considered in the case when $\CG$ is a parahoric group scheme, and then there is a whole theory devoted to them. We refer to \cite{PRS} for a survey of this theory.  
The local model in its non-naive version (flat closure of the generic fiber $M^{\rm naive}\otimes_{\CO_E} E$ of $M^{\rm naive}$ inside $M^{\rm naive}$) {\it conjecturally} only depends on the triple $\big (G, \{\mu\}, \CG\big)$ (when $\CG$ is parahoric). This theory is used to prove in the EL case, when 
$F$ is an unramified extension of $\BQ_p$, that $\CM_{\CD_{\BZ_p}}$ is flat over $\CO$ (this even holds in this case for arbitrary periodic lattice chains (G\"ortz \cite{Goertz1}), cf. \cite{PRS}, Thm.~2.3. Something similar holds in the PEL case when $B=F$ is an unramified extension of $\BQ_p$ \cite{Goertz2}, cf. \cite{PRS}, Thm.~2.9. On the other hand, when ramification occurs, then flatness may fail for  $\CM_{\CD_{\BZ_p}}$, contrary to what was conjectured in \cite{RZ}, p.~95, or in \cite{Rapo-ICM}, Remarks~1.3, (ii). 

In the context of the definition of the formal schemes 
$\CM_{\CD_{\BZ_p}}$, it is natural to consider also cases when $\CG$ is not parahoric, and this apparently also has applications 
to the study of {\it Exotic Arithmetic Fundamental Lemmas}, cf.~\cite{RapoZh}.

\end{remark}
The following question  is fundamental for all that follows. We refer to Conjecture \ref{conjnonempty} for a conjectural answer when $\CG$ is parahoric. 
\begin{question}\label{questnonempty} When is the rigid-analytic space over $\breve E$ associated to the formal scheme $\CM_{\CD_{\BZ_p}}$  non-empty?
\end{question}
\begin{remark}\label{unramifiedRZ}
 Here is what the theory of naive local models has to say about this question. We call a simple integral RZ datum {\it unramified} if the following conditions are satisfied, cf. \cite{Fargues}. In the EL case we demand that $B=F$ is an unramified field extension of $\BQ_p$. In the PEL case we demand in addition that the lattice $\Lambda$ is selfdual, i.e., $\Lambda^\vee=\Lambda$. In the case of unramified simple integral RZ data, it is easy to construct a framing object, and $\CM_{\CD_{\BZ_p}}$ is formally smooth over $\CO$. Hence Question~\ref{questnonempty} has a positive answer in this case, cf., e.g., \cite{Fargues}. 

More generally, if the naive local model in Remark \ref{locmodrem} is topologically flat over $\CO_E$, then the answer to Question \ref{questnonempty} is positive. This applies to all EL cases when $F/\BQ_p$ is unramified.  Also, all PEL cases fall into this class for which $B=F$ and where $*$ is trivial, or where $F/F_0$ is an unramified quadratic extension. Here again $F_0$ denotes the fix field of $*$ in $F$. 
\end{remark}
\begin{example}\label{exempty}
 Let us give an example where the associated rigid-analytic space is empty. The example will be of PEL type.  Let  $F_0=\BQ_p$ and take $B=F$ to be a {\it ramified} quadratic extension of $F_0$, and the involution $*$ the non-trivial Galois automorphism, and $V$ to be an $F$-vector space of dimension $2$. Then $G$ is the group of unitary similitudes, as follows.

We fix a uniformizer $\pi$ of $\CO_F$ such that $\pi^2 = \pi_0 \in \BZ_p$ (recall that $p \neq 2$ was our blanket
assumption in the PEL case). Then there exists a unique hermitian form $h$ on $V$ such that
\begin{equation*}
 (x, y) = {\rm Tr}_{F/\BQ_p} \pi^{-1} h (x, y) , \quad x, y \in V \, .
\end{equation*}
Recall that $(\, ,\, )$ is the symplectic form from the simple rational RZ data. Then the group $G$ of subsection \ref{RZdata} can be identified with the group ${\rm GU} (V, h)$ of unitary similitudes of
the hermitian vector space $(V, h)$. We note that
\begin{equation*}
G_{\rm der}  \simeq \begin{cases}
\begin{aligned}
              & \SL_{2} & \text{ when}&\text{  $(V, h)$  is split}\\
& B^1 & \text{ when }&\text{$(V, h)$ is  non-split}
\end{aligned}
             \end{cases}
\end{equation*}
Here $B^1$ denotes the norm-$1$-group of the quaternion division algebra $B$ over $\BQ_p$, and we recall that  $(V, h)$
is called split or non-split depending on whether $V$ contains an isotropic vector or not.

In the example we take $(V, h)$ to be split and fix a basis $\{ e_1, e_2 \}$ of $V$ such that the hermitian
form $h$ is given by the matrix
\begin{equation*}
 \left( \begin{array}{cc}
             0 & 1 \\
1 & 0
            \end{array} \right) \, .
\end{equation*}
We fix the conjugacy class $\{\mu \}$ by declaring that for any $\mu \in \{ \mu \}$ the eigenspaces $V_0$
and $V_1$ of $V \otimes_E \bar{\BQ}_p$ are one-dimensional isotropic subspaces (case of {\it signature} $ (1,1)$). Then
$E_{\{\mu\}} = \BQ_p$.
We note that \cite{PRLM3}, 1.b.3, b),
\begin{equation*}
 \pi_1 (G)_\Gamma \simeq \BZ / 2 \oplus \BZ \, ,
\end{equation*}
\begin{equation*}
 {\rm via} \quad g \longmapsto 
 \left( \frac{c(g)}{{\rm det} (g)} \, , {\rm val} \, c(g) \right) \, .
\end{equation*}
Consider the two elements in $B(G)_{\rm basic}$ represented by
\begin{equation*}
 b_1 = \left( \begin{array}{cc}
               0 & \pi \\
- \pi & 0
              \end{array} \right) \, , \quad
b_2 = \left( \begin{array}{cc}
               \pi & 0 \\
0 & \pi
              \end{array} \right) \, . 
\end{equation*}
Since
\begin{equation*}
 \mu^{\natural} = (0 , 1) \in \BZ/2 \oplus \BZ \, ,
\end{equation*}
we see that
\begin{equation*}
 [b_1] \in B (G, \{\mu\}) \, , \quad [b_2] \in A(G, \{\mu\})\setminus B (G, \{\mu\}) \, .
\end{equation*}
Now let us consider simple  integral RZ data $\CD_{\BZ_p}$ corresponding to an $\CO_F$-lattice
$\Lambda \subset V$ such that
\begin{equation*}
 \Lambda^{\vee} = \pi \Lambda
\end{equation*}
(since $V$ is split, such a lattice exists). Therefore for the formal moduli problem we consider objects $(X, \iota, \lambda)$, 
where $X$ is a $p$-divisible group of height $4$ and dimension $2$, and where
\begin{equation*}
 \iota : \CO_F \lra \End (X)
\end{equation*}
is such that
\begin{equation*}
{\rm char} \, (\iota (a) \vert \Lie X) =  T^2- {\rm Tr}_{F/\BQ_p}(a)  T +{\rm Nm}_{F/\BQ_p}(a)  \, , \forall a \in \CO_F ,
\end{equation*}
and where $\lambda$ is a polarization whose Rosati involution induces the non-trivial Galois automorphism
of $F/\BQ_p$ and with
\begin{equation*}
 \Ker \lambda = X [ \iota (\pi)] \, .
\end{equation*}
Consider framing objects $(\BX_1, \iota_{\BX_1}, \lambda_{\BX_1})$ and 
$(\BX_2, \iota_{\BX_2}, \lambda_{\BX_2})$ for $b_1$ resp. $b_2$  (cf. \cite{KR-alt},\S 5, case d) for the explicit construction of these framing objects).

\smallskip

\noindent{\bf Claim:} {\it 
 The formal schemes $\CM_1$,  resp. $\CM_2$,  for either choice of framing object have the following structure:
\begin{equation*}
 \CM_1 \cong \coprod_{\CL} \Spf\, W[[ t ]] \, ,\quad  \CM_2 \simeq \coprod_{\CL} \Spec\,  \BF \, ,
\end{equation*}
where $\CL$ is a countable index set. Here $\BF=\overline\BF_p$ denotes the residue field of $\mathcal{O}$ and $W=W(\BF)$ denotes the ring of Witt vectors.} 
\begin{proof}
 Let us determine $\CM_2(\BF)$. The Dieudonn\'{e} module of $\BX_2$ with its action by $\CO_F$ and its Frobenius endomorphism is given by
\begin{equation*}
 M (\BX_2) = \Lambda \otimes_{\BZ_p} W \, , \, \underline F = \pi \cdot (\id_{\Lambda} \otimes \sigma) \, ,
\end{equation*}
where the alternating form defined by the polarization is given by
\begin{equation*}
 \langle x, y \rangle = {\rm Tr}_{F/\BQ_p} \pi^{-1} h ( x , y ) \, .
\end{equation*}
Here the hermitian form $h$ is extended to $\Lambda\otimes_{\BZ_p} W$ by
\begin{equation*}
 h ( v \otimes \alpha , w \otimes \beta ) = \alpha \cdot \sigma (\beta) \cdot h (v, w) 
\in F \otimes_{\BZ_p}W \, .
\end{equation*}
We identify $F \otimes_{\BZ_p} W$ with $\breve{F}$, and write $N$ for $M (\BX_2) \otimes_{\BZ_p} \BQ_p$.
Then
\begin{equation*}
 \CM_2 (\BF) = \{ M \text{ $\CO_{\breve F}$-lattice  in $N$} \mid \pi^2 M \subset^2 \underline F M \subset^2 M , M ^\vee = \pi M \} \, .
\end{equation*}
Now,  following \cite{KR-alt}, proof of Lemma 3.2, and using the fact that $V$ is split, one proves the following points.
Let $M \in \CM_2 (\BF)$. Then 

\smallskip

\noindent(i) either $\sigma (M) = M$, i.e., $\underline F M=\pi M$,  and then for the lattice $\Lambda=\Lambda(M) := M^{<\sigma>}$
in $V$ we have
\begin{equation}\label{vertlat}
 \Lambda^{\vee} = \pi \Lambda \, ,
\end{equation}

\smallskip

\noindent (ii) or $M \neq \sigma (M)$ and $ M \subset^1 M + \sigma (M) = \sigma(M + \sigma (M))$.

\smallskip

We claim that the alternative (ii) does not occur. Indeed, in this case one can choose an $\CO_{\breve{F}}$
-basis $f_1, f_2$ for $M$ such that $f_1$ is isotropic and $\sigma (f_1) = f_1$ and such that in terms
of this basis the form $h \vert M \times M$ is given by
\begin{equation*}
 h \vert M \times M : \left( \begin{array}{cc}
               0 & \pi^{-1} \\
- \pi^{-1} & d
              \end{array} \right) \, , \, d \in \breve{F} \, .
\end{equation*}
But then
\begin{equation*}
 M + \sigma (M) = [f_0, \pi^{-1}f_0 , f_1 ] = [\pi^{-1} f_0 , f_1] \, ,
\end{equation*}
cf. \cite{KR-alt}, proof of Lemma~3.2. Hence
\begin{equation*}
 (M + \sigma (M))^{\vee} = [\pi f_0, \pi^2 f_1 ] = \pi^2 \cdot (M + \sigma (M) ) \, .
\end{equation*}
Let $\Lambda (M) := \pi \cdot (M + \sigma (M))^{<\sigma>}$. Then $\Lambda (M)$ is a lattice in $V$ which is
selfdual for $h$. The hermitian form $h$ induces on the $\BF_p$-vector space $V(M) = \Lambda (M) / \pi \Lambda (M)$ a \emph{symmetric}
bilinear form with values in $\BF_p$, which is isotropic (we are using here that $V$ is split). Furthermore, the set of
all $M'$ of type (ii) with $\Lambda (M') = \Lambda (M)$ corresponds to the set of all isotropic lines in
the $2$-dimensional $\BF$-vector space $V (M) \otimes_{\BF_p} \BF$ which are \emph{not} $\BF_p$-rational.
But this set is empty (there are only two isotropic lines and both are $\BF_p$-rational).

We denote by $\CL$ the set of lattices $\Lambda$ in $V$ satisfying the condition \eqref{vertlat}. 

In order to show the claim for $\CM_2$, we need to show that a point in $\CM_2 (\BF)$ has no deformations.
This follows from the theory of local models, cf. \cite{PRS}. In fact, the  corresponding naive local model is the disjoint union of a projective line over $\BZ_p$ and an isolated point. And the point  on the naive local model associated to any point of type (i) of $\CM_2(\BF)$ is
the isolated point (it is the unique point of the local model where the spin condition is violated, cf. \cite{PRS}, Remark~2.35). 

Now we discuss the claim for  $\CM_1$. Let $\CM$ be the universal formal deformation space  of the $p$-divisible group of a supersingular
elliptic curve over $\BF$,  and let  $Y$ be the universal object  over $\CM$. Then  the Serre construction 
\begin{equation*}
 Y \longmapsto \CO_F \otimes Y
\end{equation*}
induces an isomorphism between $\CM$  and a connected component of $\CM_1$ (use
$(\CO_F \otimes Y)^{\vee} = \pi^{-1} \CO_F \otimes Y)$. This easily implies the claim.\end{proof}

The case of $[b_2]$ shows that Question \ref{questnonempty} is non-trivial. 
\end{example}

Fix a simple  integral RZ datum, and let $\CM$ be the corresponding formal scheme over $\Spf\, \CO$. Then $\CM$  is not defined over $\Spf\, \CO_E$; however it possesses a {\it Weil descent datum} to $\Spf\, \CO_E$, in the following sense, cf.~\cite{RZ}, Def.~3.45. Let $\tau\in\Gal(\breve E/E)$ be the relative Frobenius automorphism. Let $\tau_*(\CM)$ be the formal $\CO$-scheme whose underlying formal scheme is $\CM$, but whose structure morphism is the composition 
$$
\CM\to \Spf\, \CO \buildrel{\tau^*}\over\longrightarrow\Spf\,\CO .
$$ Then a Weil descent datum on $\CM$ is an isomorphism $\tau_*(\CM)\to \CM$, comp. \cite{DOR}, p. 224. 

\subsection{Passing to the generic fiber}

Starting with an integral RZ datum $\CD_{\BZ_p}$, we obtain  the formal scheme $\CM=\CM_{\CD_{\BZ_p}}$ over $\CO$, and its {\it generic fiber} $\CM^{\rm rig}$, a rigid-analytic space over $\breve E$.

The following conjecture represents the  connection between the theory of RZ spaces and the conjectural theory of local Shimura varieties of section \ref{locShvar} below.  In most EL cases it follows from \cite{Scholzewein}. 
\begin{conjecture}\label{conjindep} Let ${\CD_{\BZ_p}}$ be an integral RZ datum such that $\CG$ is a parahoric group scheme. Then, up to isomorphism, the rigid space associated to 
 the formal scheme $\CM_{\CD_{\BZ_p}}$ with its action of $J_b(\BQ_p)$  depends  on ${\CD_{\BZ_p}}$ only via the quadruple $(G, [b], \{\mu\}, \CG)$. 
\end{conjecture}

 Let $\mathcal T$  denote the local system over $\CM^{\rig}$ defined by the $p$-adic Tate module of the universal $p$-divisible group on $\CM$, together with its additional structure (action of $\CO_B$ in the EL case; action of $\CO_B$ and polarization pairing in the PEL case). Set $\CV=\CT\otimes\BQ_p$.  
 \begin{proposition} Let $V'$  be the form of $V$ defined by any representative of the cohomology class $c([b], \{\mu\})\in H^1(\BQ_p, G)$, cf.~Proposition \ref{Fontaininequ}. 
 For the fiber of the $\BQ_p$-local system $\CV$ at any $x\in\CM_{\CD_{\BZ_p}}^{\rig}(\overline {\breve E})$ there is an isomorphism $\CV_x\simeq V'$ (respecting the additional structures on both sides).  
 \end{proposition}
 Here, in the PEL case, respecting the forms on both sides means up to a scaling factor in $\BQ_p^\times$ since, to interpret the polarization pairing as taking values in $\BQ_p$,  one has to trivialize $\BQ_p(1)$---which is unique only up to  $\BQ_p^\times$. 
 \begin{proof}
 The point $x$  is represented by a $p$-divisible group $X$ with additional structure over the ring of integers in a finite field extension of $\breve E$. This defines   a filtered isocrystal and, via the framing and $\rho$,  a pair $(b, \mu)$ with $b\in [b]$ and $\mu\in\{\mu\}$. This pair is weakly admissible, and the cocycle class measuring the difference between $\CV_x$ and $V$ is equal to $cls(b, \mu)$, cf. Proposition  \ref{Fontaininequ}. Applying Proposition  \ref{Fontaininequ}, the assertion follows.
 \end{proof}
 In the case when the group scheme $\CG$ is parahoric, we can say more.
 \begin{proposition}\label{trivpara}
  Assume that the group scheme $\CG$ over $\BZ_p$ attached to ${\CD_{\BZ_p}}$ is a parahoric group scheme. Then for the fiber of the $\BZ_p$-local system $\CT$ at any $x\in\CM_{\CD_{\BZ_p}}^{\rig}(\overline {\breve E})$ there is an isomorphism $\CT_x\simeq \Lambda$ (respecting the additional structures on both sides).  In particular, $c([b], \{\mu\})=0$. 
\end{proposition}
\begin{proof}
(cf.~proof of \cite{RZ}, Lemma 5.33.)  The cohomology class measuring the difference between these two lattices lies in $H^1(\BQ_p, \CG(\bar\BZ_p))$, and this cohomology set is trivial (here the connectedness of $\CG$ enters through Lang's theorem). 
\end{proof}
\begin{corollary}
Let ${\CD_{\BZ_p}}$ be an RZ datum such that the group scheme $\CG$ is parahoric, and such that $\CM_{\CD_{\BZ_p}}^{\rig}\neq\emptyset$. Then $[b]\in B(G, \{\mu\})$. 
\end{corollary}

Note that this corollary explains  why the generic fiber of $\CM_2$ in Example \ref{exempty} is empty. Indeed, in this case the group scheme $\CG$ is parahoric.

 \begin{remark} The first of the last two propositions is a strengthening of  \cite{RZ}, Lemma 5.33. The second proposition is contained in  \cite{RZ}, Lemma 5.33, except that in loc.~cit. the hypothesis that $\CG$ be parahoric is implicitly assumed, but is omitted from the statement.  The same omission  occurs in \cite{Rapo-ICM}, Property 3.3 (cf. also remarks after Hope 4.2 in \cite{Rapo-ICM}). 
 \end{remark}
 We conjecture that the converse of the last proposition holds. This would constitute a partial  answer to Question \ref{questnonempty}.
 \begin{conjecture}\label{conjnonempty} Let $\CD_{\BZ_p}$ be an integral RZ datum such that $[b]\in B(G, \{\mu\})$. Assume that the group scheme $\CG$ is parahoric. 
Then the  rigid-analytic space over $\breve E$ associated to the formal scheme $\CM_{\CD_{\BZ_p}}$ is non-empty. 
\end{conjecture}
\begin{remark} Example \ref{exempty} is an illustration of this conjecture: in this example $\CG$ is a parahoric group scheme, and, in the notation of loc.~cit., for $[b_1]$ (which lies in $B(G, \{\mu\})$), the generic fiber of 
$\CM_{\CD_{\BZ_p}}$ is non-empty, whereas for $[b_2]$ (which does not lie in $B(G, \{\mu\})$), the generic fiber is empty. As already pointed out in Remark \ref{unramifiedRZ}, the conjecture holds true in the case of unramified simple integral RZ data (and in many other cases as well). 

In fact, the converse should hold even in the stronger sense that the set of `special points' should be dense in the generic fiber, in the sense of Remark \ref{remunique} (ii) below. 
\end{remark}
\begin{remark} At this point one might ask why we did not impose in the definition of RZ data that $[b]\in B(G, \{\mu\})$, rather than the weaker condition $[b]\in A(G, \{\mu\})$. There are several reasons for this. First, the moduli problem leading to $\CM_{\CD_{\BZ_p}}$ can be defined in this more general case, and some of these formal schemes do define  interesting rigid-analytic spaces, cf. the example below.  Secondly, if one takes into account Remark \ref{natdef}, we see that if two integral RZ data $\CD_{\BZ_p}$ and $\CD'_{\BZ_p}$ with non-isomorphic  underlying simple rational RZ data define isogenous  framing objects, then only one of $[b]$ or $[b']$ can lie in $B(G, \{\mu\})$. It would be unnatural to exclude one of these integral RZ data, since both of them lead to the same formal scheme. 

By contrast, in the case when $\CG$ is a parahoric group scheme, it would be reasonable from the point of view of the present paper to exclude the cases where $[b]\in A(G, \{\mu\})\setminus B(G, \{\mu\})$. But even then it might well be that  the resulting formal schemes (which have empty generic fiber) are interesting. 
\end{remark}
\begin{example}\label{exnonB}
 We will describe this example in terms of the `more natural' formulation of the moduli problem, cf. Remark \ref{natdef}.  The notation $F/F_0$ with $F_0=\BQ_p$ is as in Example \ref{exempty}, and the present example will also come in pairs. We consider triples $(X, \iota, \lambda)$ just as in Example \ref{exempty}, except that now we demand that the polarization $\lambda$ is principal. As the first framing object $(\BX, \iota_\BX,  \lambda_\BX)$ we take $\BX=\CE\times\CE$, where $\CE$ denotes the $p$-divisible group of a supersingular elliptic curve over $\BF$, and where $\iota_\BX(a)={\rm diag}(a, a)$ in terms of an embedding of $\CO_F$ into $\End (\CE)$.  The polarization $\lambda_\BX$ is given by the anti-diagonal matrix with both entries equal to $1$ in terms of the natural polarization of $\CE$ which we use to identify $\CE$ with $\CE^\vee$. 

A second choice of framing object arises by  taking the same $(\BX, \iota_\BX)$ as before, but with polarization $\lambda_\BX$  given as ${\rm diag} (u_1, u_2)$, for $u_1, u_2\in \CO_F^\times$ with $-u_1u_2\notin {\rm Nm}_{F/F_0}\,F^\times$. 

We denote by $\CM_1$, resp. $\CM_2$ the formal scheme over $\CO$ that arises from the first, resp. second, choice of framing object. Then it can be shown using the theory of local models \cite{PRS}, Remark~2.35,  that both formal schemes are flat with semi-stable reduction over $\CO$. In fact, 
$$
\CM_1\simeq\coprod_{\CL}\widehat\Omega^2_{\BQ_p}\hat\otimes_{\BZ_p}\breve\BZ_p ,\quad \CM_2\simeq\CM_{\Gamma_0(p)} , 
$$
comp.~\cite{KR-alt}, \S5. Here $\widehat\Omega^2_{\BQ_p}$ denotes the Drinfeld formal scheme of dimension $2$, relative to  the $p$-adic field $\BQ_p$, and $\CL$ is a countable  index set.  Furthermore,  $\CM_{\Gamma_0(p)} $ denotes the $\Gamma_0(p)$ moduli problem. This is the RZ space relative to the simple rational RZ datum of EL type where $B=F=\BQ_p$, and $\dim V=2$, and where $\{\mu\}=\{\mu_0\}$ is the unique non-zero minuscule cocharacter, and $[b]$ is the unique basic element in $B(\GL_2, \{\mu_0\})$. The integral RZ datum is given by a {\it full} periodic lattice chain in $V$. 

In this example, there are two distinct choices of integral RZ data $\CD_{\BZ_p}$ and $\CD'_{\BZ_p}$ that can be used to construct either of these moduli schemes. Since the generic fibers of $\CM_1$ and $\CM_2$ are non-empty, we see, taking into account the previous remark, that even if $[b]\in A(G, \{\mu\})\setminus B(G, \{\mu\})$, the generic fiber of $\CM_{\CD_{\BZ_p}}$ may be non-empty. This is due to the fact that $\CG$ is not parahoric. 

We remark that these two choices of framing objects can be distinguished by their crystalline discriminants, cf. \cite{KR-padic}, \S 5, given by  $-1$ (resp. $+1$) for the first (resp., second) choice.

\end{example}

\subsection{The tower of rigid-analytic spaces}

In this subsection we fix an integral RZ datum $\CD_{\BZ_p}$ such that  $\CG$ is parahoric, and such that $\CM_{\CD_{\BZ_p}}^{\rm rig}\neq\emptyset$. 
Hence also $[b]\in B(G, \{\mu\})$, cf. Proposition \ref{trivpara}. 

\medskip

Let $K\subseteq \CG(\BZ_p)$ be a subgroup of finite index. We then associate to $K$  the rigid-analytic space 
$\BM^K=\BM_{\CD_{\BZ_p}}^K$  classifying $K$-level structures  of $\mathcal{V}$,
\begin{equation}
\CT\simeq\Lambda\, \mod K 
\end{equation}
(see, e.g.,  \cite{Fargues}, Def.~2.3.17 for the precise meaning of this term). Then  $\BM^K$ is a finite \'etale covering of $\CM_{\CD_{\BZ_p}}^{\rm rig}$ and, for $K_0=\CG(\BZ_p)$, we have $\BM^{K_0}=\CM_{\CD_{\BZ_p}}^{\rig}$. One can extend this definition to obtain rigid-analytic spaces $\BM^K$ associated with all open compact subgroups $K\subseteq G(\mathbb{Q}_p)$ (see \cite{RZ}, 5.34).

We obtain in this way a tower of rigid-analytic spaces $\{\BM^K\}$ over $\breve E$. In particular, $\BM^K=\BM^{K'}/(K/K')$,  if $K'$ is a normal subgroup of $K$. It is naturally endowed with an action (from the right) of $G(\mathbb{Q}_p)$  via Hecke correspondences  (see \cite{RZ}, 5.34 for the precise meaning of this term, comp. \cite{Del-Bourb}, Remarque 3.2). Note that this  is an action on the whole tower and not on the individual spaces (an element $g\in G(\mathbb{Q}_p)$ maps $\BM^K$ to $\BM^{g^{-1}Kg}$). 

\begin{properties}\label{rempropmgk}
{\rm The tower of rigid-analytic spaces $\BM_{\CD_{\BZ_p}}^K$ has the following properties.}
\begin{altenumerate}
\item It only depends on the simple rational RZ data $\CD$ underlying the integral data $\CD_{\BZ_p}$, cf. \cite{RZ}, 5.38 and 5.39. We therefore use below  the notation $\BM_{\CD}^K$ for $\BM_{\CD_{\BZ_p}}^K$.

\item The adic space corresponding to $\BM_{\CD}^K$  is smooth and partially proper over $\breve E$, cf. \cite{Fargues}, Lemme~2.3.24; the transition morphisms $\BM_{\CD}^{K'}\to\BM_{\CD}^K$ for $K'\subset K$ are finite and etale. 

\item The action of  $J(\mathbb{Q}_p)$  on $\CM_{\CD_{\BZ_p}}$ induces an action  on $\BM_{\CD}^{K_0}=\CM_{\CD_{\BZ_p}}^{\rig}$. This action lifts to an action (from the left) of $J(\BQ_p)$ on  $\BM_{\CD}^K$, for each $K$, cf. \cite{RZ}, 5.35.  When the center $Z$ of $G$ is identified with a subgroup of $J$, then the left action of an element  $z\in J(\BQ_p)$ with $z\in Z(\BQ_p)$   coincides with the right action of $z\in G(\BQ_p)$, cf. \cite{RZ}, Lemma 5.35.  

\item The  Weil descent datum for $\CM_{\CD_{\BZ_p}}$ over $\mathcal{O}_E$ induces a Weil descent datum on $\BM_{\CD}^{K_0}=\CM_{\CD_{\BZ_p}}^{\rig}$ over $E$ (see \cite{RZ}, 5.43 for the (obvious)  explanation of this term). This lifts to a system of compatible Weil descent data for all $\BM_{\CD}^K$ over $E$, cf. \cite{RZ}, 5.47. 
\end{altenumerate}
\end{properties}

We continue to denote by $\CM$ the formal scheme associated to the integral RZ datum  ${\CD_{\BZ_p}}$.  Denote as above by $\mathcal M^{\rig}$ its generic fiber,  and let $(X,\rho)$ be  the universal $p$-divisible group over $\mathcal M^{\rig}$, 
with additional structure and equipped with the universal quasi-isogeny. Then  $\rho$ induces an isomorphism 
$$V\otimes_{\BQ_p}\CO_{\CM^{\rig}}\cong M(X)\otimes_{\CO_{\CM}}\mathcal O_{\mathcal M^{\rig}},
$$
 where $M(X)$ denotes the Lie algebra of the universal vector extension of $X$. The surjection  $M(X)\to\Lie\, X$ thus yields a filtration on 
 $V\otimes_{\BQ_p}\CO_{\CM^{\rig}}$ which  corresponds to a  morphism $\breve \pi: \mathcal M^{\rig}\rightarrow \breve{\mathcal F}^{\rig} ,$
 which factors through the period domain $\CF^{\rm wa}=\CF(G, b, \{\mu\})^{\rm wa}$ inside $\breve{\mathcal F}^{\rig} $. Here  the period domain is attached to the triple $(G, b, \{\mu\})$ defined by the simple rational RZ datum $\CD$, where we choose the same element $b\in[b]$ as in the definition of the formal scheme $\CM$ and as in the definition of $J$.  This morphism of rigid-analytic spaces over $\breve E$ is  the {\it period morphism} associated to the moduli space $\mathcal M$. The period morphism extends to a compatible system of morphisms 
  \begin{equation}
 \breve \pi^K: \BM^{K}_\CD\rightarrow \breve{\mathcal F}^{\rig},
 \end{equation}
 for varying $K\subset G(\BQ_p)$. 
 
\begin{properties}\label{properiod}
 {\rm We list some properties of the period morphisms.}
\begin{altenumerate}
\item The morphisms $\breve \pi^K$ are \'etale and partially proper for all $K$, cf. \cite{RZ}, Prop. 5.17, \cite{Fargues}, Lemma~2.3.24. The fiber of $\breve \pi^K$ through a point $x\in\BM^K$ may be identified with $G(\BQ_p)/K$, cf. \cite{RZ}, Prop. 5.37.

\item The morphisms $\breve \pi^K$ are $J(\mathbb{Q}_p)$-equivariant. Here $J(\BQ_p)$ acts on $\breve\CF^{\rm rig}$ as described in section \ref{perdom}.

\item The morphisms   $\breve \pi^K$ factor through $\breve\CF^{\rm wa}$. 

\item The morphisms   $\breve \pi^K$  are \emph{not} compatible with the Weil descent data of source and target. Here the source is equipped with the Weil descent datum of Properties \ref{rempropmgk}, (iv); the target is equipped with the Weil descent datum coming from the fact that $\breve\CF$ is defined over $E$. However, let $\Delta$ be the dual abelian group of $X^*(G_{\rm  ab})^\Gamma$. We denote by the same symbol the corresponding discrete rigid-analytic space. Then there is a natural morphism $\kappa^K: \BM^K\to\Delta$ which is equivariant for the actions of $G(\BQ_p)$ and $J(\BQ_p)$, cf. \cite{RZ}, 3.52. There is a natural Weil descent datum on the product $\breve\CF\times\Delta$, cf. \cite{RZ}, 5.43, and the product morphism
$$
(\breve \pi^K, \kappa^K): \BM^K\to \breve\CF^{\rm rig}\times\Delta
$$
is compatible with the Weil descent data on source and target, cf. \cite{RZ}, 5.46. 
\end{altenumerate}
\end{properties}
Let $\CF^{\rm a}$ be the image of $\breve\pi^K$ as an adic space (this is independent of $K$). Then $\CF^{\rm a}\subseteq\CF^{\rm wa}$. This inclusion may be proper, even though both adic spaces have the same sets of points with values in a finite extension of $\breve E$. We refer to papers by Faltings \cite{Falt}, Hartl \cite{Hartl} and Scholze/Weinstein  \cite{Scholzewein} for various descriptions of $\CF^{\rm a}$ (the {\it admissible locus}), cf. also \cite{DOR}, ch. XI, \S4.

So far very little is known about the geometry of the spaces $\BM^K$. However, let us give a conjectural description of their sets of geometrically connected components. Note that these are of particular interest as they are closely related to the highest degree cohomology groups with compact support of the spaces. 

\begin{definition}\label{defhnregular}
A pair $([b],\{\mu\})$ such that $[b]\in B(G,\{\mu\})$ is called HN{\it-reducible} if there exists a proper parabolic subgroup $P$ with Levi factor $L$ (both defined over $F$), a representative $\mu'\in \{\mu\}$ which factors through $L$ and an element $b'\in [b]\cap L(\breve F)$ such that
\begin{altitemize}
\item $[b']_L\in B(L,\{\mu'\}_L)$ where $[b']_L$ is the $L$-$\sigma$-conjugacy class of $b'$ and where $\{\mu'\}_L$ is the $L$-conjugacy class of $\{\mu'\}$.
\item In the action of $\mu'$ and of $\nu_{b'}$ on $\big(\Lie\, R_u(P)\big)\otimes_F\bar F$, only non-negative characters occur. Here $R_u(P)$ denotes the unipotent radical of $P$. 
\end{altitemize}
The pair $([b],\{\mu\})$ is called HN{\it -irreducible} if it is not HN-reducible.
\end{definition}
\begin{remark}
In the special case that $G$ is unramified choose a maximal torus $T$ and Borel $B$ containing $T$, both defined over $F$. In the above definition we may replace $b'$, $\mu'$, $L$, and $P$  by ($\sigma$-)conjugates and thus assume that in addition $P$ is a standard parabolic subgroup, $L$ its Levi factor containing $T$, and that $\mu'\in X_*(T)_{\dom}$ and $\nu_{b'}\in X_*(T)_{\mathbb{Q},\dom}$. Thus, in this case, a pair $([b],\{\mu\})$ such that $[b]\in B(G,\{\mu\})$ is HN-reducible if there exists a proper standard parabolic subgroup $P$ with Levi factor $L$ containing $T$, such that for the dominant representative $\mu'=\mu_{\dom}\in X_*(T)_{\dom}$ and a representative $b'\in[b]\cap L (\breve F)$ such that $\nu_{b'}\in X_*(T)_{\mathbb{Q},\dom}$ we have $[b']_L\in B(L,\{\mu'\}_L)$. We denote this dominant Newton point (which only depends on $[b]$, not on $b'$) by $\nu_{[b]}=\nu_{b'}$. By \cite{CKV}, Theorem 2.5.6 the above condition for HN-reducibility is equivalent to the condition that for some proper standard Levi subgroup $L$ of $G$ we have that $\overline\mu_{\dom}-\nu_{[b]}$ is a non-negative rational linear combination of coroots of $A$ in $L$. Here $A\subset T$ is the maximal split subtorus, and $\overline {\mu}_{\dom}$ denotes the image of $\mu_{\dom}$ in $X_*(A)_{\mathbb{Q}}=X_*(T)_{\mathbb{Q},\Gamma}$.

Let us make this condition more explicit in the case that $G=\GL_n$. We choose the upper triangular matrices as Borel subgroup and let $T$ be the diagonal torus. Then $X_*(T)\cong \mathbb{Z}^{n}$ and $x=(x_i)\in X_*(T)$ is dominant if $x_1\geq\dotsm\geq x_n$. Hence $\nu_{[b]}\in X_*(T)_{\mathbb{Q},\dom}$ is an ordered $n$-tuple of rational numbers, and $\mu_{\dom}=\overline\mu_{\dom}$ an ordered $n$-tuple of integers. We associate with each such tuple $(y_i)\in X_*(T)_{\mathbb{Q},\dom}$ the corresponding polygon which is defined as the graph of the piecewise linear continuous function $[0,n]\rightarrow \mathbb{R}$ mapping $0$ to itself and with slope $y_i$ on $[i-1,i]$. For $\nu_{[b]}$ we obtain the so-called {\it Newton polygon}, for $\mu_{\dom}$ the {\it Hodge polygon}. The condition that $[b]\in B(G,\{\mu\})$ is the equivalent to the condition that the Hodge polygon lies above the Newton polygon and that they have the same endpoint. Each standard Levi subgroup $L$ of $G$ is of the form $\GL_{n_1}\times \dotsm\times \GL_{n_l}$, embedded diagonally into $\GL_n$ and where $(n_1,\dotsc, n_l)$ is a partition of $n$. The pair $([b],\{\mu\})$ is HN-reducible if and only if for some such subgroup (with $l>1$) the two polygons coincide also at $n_1,n_1+n_2,\dotsc$. An easy combinatorial argument shows that this condition is equivalent to the classical condition by Katz \cite{Katz} that the Hodge polygon and the Newton polygon have some break point of the Newton polygon in common, or  coincide and do not have any break point. Similar descriptions are possible for the other classical groups.
\end{remark}

Let $\pi_0(\BM^K)$ denote the set of geometric connected components, i.e.,  the set of connected components of $\BM^K\times_{{\rm Sp}\, \breve E}{{\rm Sp}\, \overline{\breve E}}$. 
\begin{conjecture}\label{conjconncomp}
Assume that $G_{\der}$ is simply connected. Let $D=G_{\rm ab}=G/G_{\der}$ denote the maximal abelian quotient of $G$ and $\delta:G\rightarrow D$ the projection map. Then there is a morphism of towers 
$$(\BM^K)_K\rightarrow (\BM_{D,\delta (\mu)}^{\delta(K)})_K=(D(\mathbb{Q}_p)/\delta(K))_K$$
which is compatible with the group actions on the towers. Here the right hand side is the tower  associated with the torus $D$ as in subsection \ref{secexlsh}, (i) below. Furthermore, if $([b],\{\mu\})$ is HN-irreducible, then this morphism induces bijections 
$$\pi_0(\BM^K)\cong D(\mathbb{Q}_p)/\delta (K)$$
 which are compatible for varying  $K$.
\end{conjecture}
For unramified simple RZ spaces of EL type or  PEL type (as in Remark \ref{unramifiedRZ}), this conjecture has been shown by Chen, \cite{C}, compare also \cite{CKV}, section 5. 
\begin{remark}
Let us consider the case when  $G$ is unramified. Chen's method to determine $\pi_0(\BM^{K_0})$ for $K_0=\CG(\mathbb{Z}_p)$ is to determine the set of connected components of the reduced subscheme $\CM_{\rm red}$ underlying the formal scheme $\mathcal{M}$, and then to use the formal smoothness of the moduli space. This is based on de Jong's theorem \cite{J1}, Thm.~7.4.1, that the set of connected components of $\CM^{\rm rig}$ is in bijection with the set of connected components of $\CM_{\rm red}$, provided that $\CM$ is (formally) normal. The sets of connected components of these underlying reduced subschemes can be described in terms of affine Deligne-Lusztig varieties (compare \cite{CKV}, section 5). 

The discussion in \cite{CKV}, Section 2.5,  together with Theorem 1.1 of loc.~cit., shows that,  if $([b],\{\mu\})$ is HN-reducible, then $\pi_0(\BM^{K_0})$ has a description in terms of the  Levi subgroup $L$ of $G$ occurring in Definition \ref{defhnregular}. In particular,  if $([b],\{\mu\})$ is HN-reducible, $\pi_0(\BM^{K_0})$ is never given by the formula in Conjecture \ref{conjconncomp} (for $K=K_0$ the RHS is equal to  $\pi_1(G)_{\Gamma}).$ An extreme example of this observation are moduli spaces of ordinary $p$-divisible groups of given dimension $d$ and codimension $c$: Then $G=\GL_{c+d}$, and the reduced subscheme of the moduli space is discrete and one has a bijection
\begin{equation*}
\pi_0(\BM^{K_0})\cong\GL_c(\mathbb{Q}_p)/\GL_c(\mathbb{Z}_p)\times \GL_d(\mathbb{Q}_p)/\GL_d(\mathbb{Z}_p).
\end{equation*}
 If $G$ is non-split,  there are also examples where the pair $([b],\{\mu\})$ is HN-reducible but where the associated $p$-divisible groups are bi-infinitesimal.
\end{remark}

\section{Local Shimura varieties}\label{locShvar} 

The following definition is the local analogue of a Shimura datum in the sense of Deligne \cite{Del-Bourb}.
\begin{definition} Let $F$ be a finite extension of $\BQ_p$. 
 A {\it local Shimura datum} (over $F$) is a triple  $(G, [b], \{\mu\})$ consisting of the following data. The first entry  is a reductive algebraic group  $G$ over $F$.  The second entry $[b]$ is a $\sigma$-conjugacy class in $B(G)$. The third entry $\{\mu\}$ is a geometric conjugacy class of cocharacters of  $G$ defined over some fixed algebraic closure $\bar F$ of $F$. We make the following two assumptions.
 \begin{altenumerate}
 \item  $\{\mu\}$ is minuscule,
 
 \item $[b]\in B(G,\{\mu\})$.
 \end{altenumerate}
\end{definition}
Associated to a local Shimura datum are the following data.
 \begin{altitemize}
 \item the {\it reflex field} $E=E(G, \{\mu\})$ (the field of definition of $\{\mu\}$ inside a fixed algebraic closure $\bar F$). Let $\breve E$ be the completion of the maximal unramified extension of $E$ in $\bar F$. 
 
 \item the algebraic group $J=J_b$ over $F$, for   $b\in [b]$. 
  
 \smallskip
 
 In the sequel, we fix $b\in [b]$. 
 \end{altitemize}
 \subsection{The conjecture}
 We have seen in the previous section that local Shimura data over $\BQ_p$  arise from  simple rational RZ data of type EL or PEL. In this section we formulate the `conjecture' (in the vague sense) that  to any local Shimura datum is associated a tower of rigid-analytic spaces similar to those explained in the previous section in the context of RZ spaces. 
 
 Indeed, one would like to associate to a local Shimura datum $(G,  [b], \{\mu\})$ over $F$ a tower of rigid-analytic spaces over ${\rm Sp}\, \breve E$, 
 \begin{equation}\label{eqtower}
\BM(G, [b], \{\mu\})= \{\BM^{K}\mid K\subset G(F)\} \, ,
 \end{equation}
 where $K$ ranges over all open compact subgroups of $G(F)$, with the following properties. 
 
 \begin{altenumerate}
 \item each $\BM^{K}$ is equipped with an action of $J(F)$. 
 \item the group $G(F)$ operates on the tower as a group of Hecke correspondences. 
 
  \item the tower is equipped with a Weil descent datum down to $E$. 
  
 \item there exists a compatible system of \'etale and partially proper {\it  period morphism(s)} $\breve\pi^{K}: \BM^{K}\lra \breve{\bf\CF} (G, b, \{\mu\})^{\rm wa}$ which is equivariant for the action of $J(F)$ and which is the first component of a $J(F)\times G(F)$-equivariant morphism of towers of rigid-analytic spaces
 \begin{equation}\label{periodmorph}
 (\breve \pi^{K}, \kappa^{K}): \BM^{K}\to {\bf\CF} (G, b, \{\mu\})^{\rm wa}\times\Delta
\end{equation}
 compatible with Weil descent data,  analogous to Properties \ref{properiod}, (iv).  The fiber of $\breve \pi^{K}$ through a point of $\BM^{K}$ is in bijection with $G(F)/K$, compatibly with changes in $K$.  In fact, there should be an open adic subspace ${\bf\CF} (G, b, \{\mu\})^{\rm a}$ of ${\bf\CF} (G, b, \{\mu\})^{\rm wa}$ stable under the action of $J(F)$ and a $J(F)$-equivariant $G(F)$-torsor such that the tower $\{\BM^K\}$ arises from the system of  level-$K$-structures on this local system. 
 \end{altenumerate}

The  tower would be called the {\it local Shimura variety} associated to the local Shimura datum $(G, [b], \{\mu\})$.  We do not know how to construct  local Shimura varieties in the most general case (nor do we know how to describe the corresponding open subspaces ${\bf\CF} (G, b, \{\mu\})^{\rm a}$ of ${\bf\CF} (G, b, \{\mu\})^{\rm wa}$), but we conjecture their existence, comp. \cite{Rapo-ICM}, Hope 4.2. 
\begin{remark}
We alert the reader to a   conflict of notation between section \ref{rzspaces} and the other sections of this paper. In section \ref{rzspaces},  the notation $F$ appears as the center of the algebra $B$; but the corresponding local Shimura datum is over $\BQ_p$. In the present section, the local Shimura datum is over $F$.

A local Shimura datum $(G, [b], \{\mu\})$ over $F$  defines a local Shimura datum $(G', [b'], \{\mu'\})$ over $\BQ_p$. Indeed, fix an embedding $\iota_0: F\to \bar\BQ_p$. Let $G'={\rm Res}_{F/\BQ_p}\, G$. Then under the Shapiro isomorphism there is an identification  $B(\BQ_p, G')=B(F, G)$, cf. \cite{K}, 1.10. We take $[b']=[b]$ under this identification. Finally, giving a conjugacy class of cocharacters of $G'$ is equivalent to giving a family of conjugacy classes of cocharacters of $G$, one for each embedding $\iota: F\to \bar\BQ_p$. For $\{\mu'\}$ we take the family consisting of $\{\mu\}$ at the fixed embedding $\iota_0$, and the trivial cocharacter at all other embeddings, comp. (vi) of Properties \ref{proplocSh} below. To some extent the case of general $F$ can thus be reduced to the case where $F=\BQ_p$, just as in Deligne's set-up for global Shimura variety one starts with a reductive algebraic group over $\BQ$. However, this approach also has its disadvantages (e.g., Bruhat-Tits theory does not interact so well with restriction of scalars). Since starting with a general field $F$ also has its own aesthetic appeal,   we decided to stick with the present set-up. 
\end{remark}
\begin{properties}\label{proplocSh}
 {\rm We expect the following additional properties of local Shimura varieties. }

\begin{altenumerate}
\item Each member $\BM^{K}$ is a smooth and partially proper rigid-analytic space of dimension $\langle \mu_{\rm dom}, 2\rho\rangle$. Here  $\mu_{\rm dom}$ denotes the dominant representative of $\{\mu\}$ in $X_*(T)$ for a choice of a maximal torus $T$ in $G_{\bar F}$ and a Borel subgroup containing $T$, and  $\rho$ denotes the half-sum of all positive roots of $T$. Furthermore, the transition morphisms $\BM^{K_1}\to\BM^{K_2}$ for $K_1\subset K_2$ are finite and \'etale. 

\item The action of $J(F)$ on $\BM^{K}$ is continuous, when $\BM^{K}$ is considered as an analytic space in the sense of Berkovich, comp. \cite{Fargues}, Cor. 4.4.1. Equivalently, in terms of the adic space $\BM^{K}$, the action of $J(F)$ is continuous on the underlying topological space and,  for any open affinoid open $U$, denoting by $J_U$ the stabilizer of $U$ in $J(F)$, the map
\begin{equation*}
J_U\to \End_{\rm Banach}\big(\Gamma(U, \CO)\big)
\end{equation*} is continuous.

\item There exists an open subset $U$ of the adic space $\BM^{K}$ which, together with all its translates under $J(F)$,  covers the whole space, and which has the following property: there exists a finite covering of $U$ by open subsets $U_i$ for which there exist quasi-compact separated adic spaces $V_i$ and global functions $f_{ij}$ on $V_i$ ($ j=1,\ldots, n_i$) such that $U_i$ is defined inside $V_i$ by 
$$
U_i=\{x\in V_i\mid \vert f_{ij}(x)\vert<1,  j=1,\ldots, n_i\} . 
$$

\item If  $(G_1, [b_1], \{\mu_1\})$, resp. $(G_2, [b_2], \{\mu_2\})$, are local Shimura data, and if $f: G_1\to G_2$ is a compatible homomorphism in the obvious sense, then one should have compatible morphisms of rigid-analytic spaces (compatible with Weil descent data)
\begin{equation*}
\BM(G_1, [b_1], \{\mu_1\})^{K_1}\to \BM(G_2, [b_2], \{\mu_2\})^{K_2}\times_{{\rm Sp}\, \breve E_2} {{\rm Sp}\, \breve E_1},
\end{equation*}
provided that $f(K_1)\subset K_2$.  Furthermore, if $f$ is a closed embedding, then  these morphisms should be closed embeddings  when $K_1=K_2\cap G_1(F)$. 

\item If $(G_1, [b_1], \{\mu_1\})$, resp. $(G_2, [b_2], \{\mu_2\})$, are local Shimura data, let $G=G_1\times G_2$ and complete $G$ to the product local Shimura datum $(G, [b], \{\mu\})$, in the obvious sense. Then the reflex field $E=E(G, [b], \{\mu\})$ is the composite of the reflex fields $E_i=E(G_i, [b_i], \{\mu_i\})$, for $i=1, 2$. For $K$ of the form $K=K_1\times K_2$, and with obvious notation, one should have compatible isomorphisms (compatible with Weil descent data)
\begin{equation*}
\BM^{K}\simeq\big(\BM_1^{K_1}\times_{{\rm Sp}\, \breve E_1} {{\rm Sp}\, \breve E}\big)\times_{{\rm Sp}\, \breve E} \big(\BM_2^{K_2}\times_{{\rm Sp}\, \breve E_2} {{\rm Sp}\, \breve E}\big) .
\end{equation*}
In (iv) and (v), the morphisms, resp. the isomorphisms,  should be compatible with the period morphisms. 

\item Let $F'/F$ be a finite extension, and let $G={\rm Res}_{F'/F} (G')$. Fix an $F$-embedding $F'\to \bar F$, and use this to write 
$$
G\otimes_F\bar F=\prod_{\Gal(\bar F/F)/\Gal (\bar F/F')} G'\otimes_{F'}\bar F . 
$$ Fix a minuscule geometric conjugacy class $\{\mu'\}$ of cocharacters of $G'$, and let $\{\mu\}$ be the minuscule geometric conjugacy class  of cocharacters of $G$ given by $(\{\mu'\}, \mu_0,\ldots,\mu_0)$. Here the first entry corresponds to the fixed embedding of $F'$ into $\bar F$, and $\mu_0$ denotes the trivial cocharacter. Note that there is a canonical identification $B(G', \{\mu'\})=B(G, \{\mu\})$, cf. \cite{K2}, (6.5.3). We complete the pairs $(G', \{\mu'\})$, resp. $(G, \{\mu\})$, into local Shimura data over $F'$, resp. over $F$, by fixing $[b']\in B(G', \{\mu'\})$, resp. $[b]\in B(G, \{\mu\})$ that correspond to each other under this bijection. Note that the local Shimura fields $E(G', \{\mu'\})$ and $E(G, \{\mu\})$ are identical. One should have an isomorphism of towers (compatible with all structures)
\begin{equation*}
\BM(G', [b'], \{\mu'\})^{K'}\simeq \BM(G, [b], \{\mu\})^K ,
\end{equation*}
where $K'=K$ under the identification $G'(F')=G(F)$. 

\end{altenumerate}
\end{properties}
\begin{remark}\label{remunique}
The obvious deficiency in the above conjecture lies in the lack of uniqueness in the definition of a local Shimura variety, i.e., we do not know how to characterize the tower attached to a local Shimura datum. In this regard we make the following comments.

(i) It seems reasonable to expect such a uniqueness statement if we restrict to local Shimura data $(G, [b], \{\mu\})$ which can be embedded in the sense of Properties \ref{proplocSh}, (iv) into a local Shimura datum of the form $(\GL_n, [b'], \{\mu'\})$ for suitable $n$ and $\{\mu'\}$ (for which the RZ construction defines a corresponding local Shimura variety). Such local Shimura data are called of {\it local Hodge type}. This is the local analogue of a {\it Shimura variety  of Hodge type}. In this case, the {\it admissible locus} $\CF(G, b, \{\mu\})^{\rm a}$ in the sense of Hartl or Faltings is defined, cf.~\cite{DOR},  ch. XI, \S4, and the period morphisms $\breve \pi^{K}$ should induce {\it surjective} morphisms 
$$
\breve\pi^{K} : \BM^{K}\to \CF(G, b, \{\mu\})^{\rm a} , 
$$
such that the fiber through a point is equal to $G(F)/K$, and such that the tower $\BM^{K}$ is embedded into the tower for $(\GL_n, [b'], \{\mu'\})$ (base-changed to $\breve E$), in the sense of property (iv) above---and these properties should characterize the local Shimura variety uniquely. 

(ii) It should also be possible to characterize local Shimura varieties as closed subvarieties of a bigger local Shimura variety by its `special points', in analogy with the classical theory, cf.~\cite{Del-Bourb}, 3.15. More precisely, in the situation of (iv) of Properties \ref{proplocSh}, assume that $f:G_1\to G_2$ is a closed embedding. Suppose that $\BM(G_2, [b_2], \{\mu_2\})$ has been defined, in such a way that for any inclusion $(T, [b_T], \{\mu_T\})\hookrightarrow (G_1, [b_1], \{\mu_1\})$, where $T$ is a maximal torus of $G_1$, we have a closed immersion of local Shimura varieties
\begin{equation*}
\BM(T, [b_T], \{\mu_T\})^{K_T}\to\BM(G_2, [b_2], \{\mu_2\})^{K_2}\times_{{\rm Sp}\, \breve E_2}{\rm Sp}\, \breve E_T
\end{equation*}
where $E_T=E(T, \{\mu_T\})$ and where $K_T=K_2\cap T(F)$. Here the local Shimura variety $\BM(T, [b_T], \{\mu_T\})^{K_T}$ is defined in subsection \ref{secexlsh}, (i) below. Let $K_1=K_2\cap G_1(F)$. 
Then $\BM(G_1, [b_1], \{\mu_1\})^{K_1}$ should be characterized as the minimal Zariski-closed subspace 
\begin{equation*}
\BM^{K_1}\subset \BM(G_2, [b_2], \{\mu_2\})^{K_2}\times_{{\rm Sp}\, \breve E_2}{\rm Sp}\, \breve E_1
\end{equation*}
 such that    $\BM(T, [b_T], \{\mu_T\})^{K_T}\subset\BM^{K_1}\times_{{\rm Sp}\, \breve E_1}{\rm Sp}\, \breve E_T$ for any $(T, [b_T], \{\mu_T\})$ as above. 
\end{remark}

\begin{example} As a typical example of a local Shimura datum  $(G, [b], \{\mu\})$ not of local Hodge type, one may take $G={\rm PGL}_n$, and $\{\mu\}$ a non-trivial minuscule cocharacter, and $[b]\in B(G, \{\mu\})$ arbitrary. 

\end{example}
\begin{remark} \label{commondef}
Let us comment on the conditions (i) and (ii) imposed on the triple  defining a local Shimura datum. 

The condition that $\{\mu\}$ be minuscule comes from the fact that the conjecture on the existence of a $\BQ_p$-local system on open adic subspaces of period domains $\CF(G, b, \{\mu\})^{\rm wa}$ in \cite{RZ}, 5.54 becomes false outside the case when $\{\mu\}$ is minuscule, as was pointed out independently by L.~Fargues \cite{Fargues-counter} and by P.~Scholze, comp. \cite{Kedlaya-ICM}. In fact, this hypothesis makes the theory of local Shimura varieties more challenging since it precludes the existence of a tannakian formalism (at least in the naive sense). 

The condition that $[b]$ lie in $B(G, \{\mu\})$ is made for aesthetic reasons, and could be weakened to $[b]\in A(G, \{\mu\})$. Indeed, let us assume this weakened condition, and fix  an unramified  cocycle representative $\eta$ of the cohomology class $c([b], \{\mu\})\in H^1(F, G)$. Let $G'$ be the inner form defined by  the image of $\eta$ in $G_{\rm ad}$. The two groups $G(\breve F)$ and $G'(\breve F)$ can be identified, so that $b$ corresponds to an element $b'\in G'(\breve F)$. Similarly, $\{\mu\}$ corresponds to a geometric conjugacy class of cocharacters $\{\mu'\}$ of $G'$. Then it follows from \eqref{twistcomp} that $[b']\in B(G', \{\mu'\})$. Therefore, the local Shimura variety attached to $(G',  [b'], \{\mu'\})$ may be used to define the local Shimura variety for $(G, [b], \{\mu\})$. Note that the groups $J_b(F)$ and $J_{b'}(F)$, which act on the individual members of the tower, can be identified, but that the group of Hecke correspondences $G'(F)$ acting on the tower is now an inner twist of the original group $G(F)$. This is the point of view adopted in \cite{Rapo-ICM}, \S 4. It has the advantage that it matches the definition of rational RZ data (recall from subsection \ref{RZdata} that we imposed there the weaker condition  $[b]\in A(G, \{\mu\})$). It has the disadvantage that the notation becomes more confusing due to the appearance of the inner form $G'$ of $G$. Since in the end the objects searched for are in effect the same, we opted for the stronger hypothesis $[b]\in B(G, \{\mu\})$.

\end{remark}
\begin{remark}\label{relRZlocsh}
 A simple rational RZ datum $\CD$ with $[b]\in B(G, \{\mu\})$ defines a local Shimura datum over $\BQ_p$, and the  formal schemes $\CM_{\CD_{\BZ_p}}$ give rise to a tower of rigid-analytic spaces as required in \ref{eqtower}, provided that the generic fiber of  $\CM_{\CD_{\BZ_p}}$ is non-empty, comp. Conjecture \ref{conjnonempty}. However, not all properties enumerated above are proved for such towers (when they make sense in this context), although we conjecture them to hold. For instance, property (iii) does not seem to be known in general, and property (iv) is closely related to Conjecture \ref{conjindep}. 

\begin{comment} The RZ construction sometimes accomodates elements $[b]\in A(G, \{\mu\})$ that do not belong to $B(G, \{\mu\})$. To illustrate this, consider the element $[b_2]\notin B(G, \{\mu\})$ of Example \ref{exempty}. In this case the formal scheme has empty generic fiber, which would lead one to believe that the corresponding member $\BM^{K_0}$ of the local Shimura variety tower is empty. However, this is simply not the correct definition of this member of the local Shimura variety in this case. For this triple $(G,[b],\{\mu\})$ one rather wants to consider the tower defined by one of the two formal schemes of Example \ref{exnonB} (the one which is a disjoint union of copies of the Drinfeld formal scheme). The member  $\BM^{K_0}$ is obtained by forming the quotient of the generic fiber of the formal scheme in Example \ref{exnonB} by a finite group. The point is that the formal scheme $\CM_{\CD_{\BZ_p}}$ of Example \ref{exempty} is not an integral model of $\BM^{K_0}$.
\end{comment}
\end{remark}
 In the context of local Shimura varieties, the duality conjecture of \cite{RZ}, 5.54 becomes easy to state (at least with the same imprecision as in loc.~cit., cf.~below). Let $(G, [b], \{\mu\})$ be a local Shimura datum such that $[b]\in B(G, \{\mu\})_{\rm basic}$. Then the {\it dual local Shimura datum} $(G^\vee, [b^\vee], \{\mu^\vee\})$ is defined as follows. Let $b\in [b]$ and set $G^\vee=J_b$. Then there are canonical identifications (of point sets, resp. of algebraic groups), 
 $$
 G^\vee(\breve F)=G(\breve F) ,\quad \quad G^\vee_{\bar F}=G_{\bar F} .
 $$
 Then set $[b^\vee]=[b^{-1}]$ and $\{\mu^\vee\}=\{\mu^{-1}\}$ under these identifications. We note that $[b^\vee]$ is a basic element of $B(G^\vee)$, and that the dual of $(G^\vee, [b^\vee], \{\mu^\vee\})$ is the original datum $(G, [b], \{\mu\})$. The group $G^\vee(F)$ acts on the local Shimura variety $\BM(G, [b], \{\mu\})$ preserving each individual member of the tower; there is a similar action of   $G(F)$ on the local Shimura variety $\BM(G^\vee, [b^\vee], \{\mu^\vee\})$.  On the other hand, the group $G(F)$ acts on the local Shimura variety $\BM(G, [b], \{\mu\})$ as a group of Hecke correspondences of the tower; there is a similar action of $G^\vee(F)$ on the local Shimura variety $\BM(G^\vee, [b^\vee], \{\mu^\vee\})$.
 
  \begin{conjecture}\label{dualityconj}
 There exists an isomorphism `in the limit'
 $$
 \varprojlim\nolimits_K \BM((G, [b], \{\mu\}))^K\simeq  \varprojlim\nolimits_{K^\vee} \BM((G^\vee, [b^\vee], \{\mu^\vee\}))^{K^\vee} ,
 $$
 compatible with the actions of $G^\vee(F)\times G(F)$ on both sides.  
 \end{conjecture}
 The imprecision here lies in the term `in the limit'; however, as in \cite{Scholzewein}, the theory of perfectoid spaces should enable one to make sense of it. More precisely, following \cite{Scholze-surv}, \S 5, there should exist a perfectoid space $\CM$ over $\hat{\bar E}$ with a continuous action of $G(F)\times G^\vee(F)$ which is equivariantly equivalent in the sense of loc.~cit. to the base change to $\hat{\bar E}$ of both limits occurring in Conjecture \ref{dualityconj}. Note that $\CM$ is unique if it exists. 

\begin{remark}
In the spirit of the uniformization theorem \cite{RZ}, Thm. 6.36, there should be a relation between global Shimura varieties and their local analogues. Let $({\bf G}, \{h\})$ be a global Shimura datum. Hence ${\bf G}$ is a reductive group over $\BQ$, and we denote by $\{\mu_h\}$ the conjugacy class of cocharacters of ${\bf G}$ over $\bar\BQ$ that corresponds to $\{h\}$. The global Shimura datum defines a Shimura variety $\BM({\bf G}, \{h\})=\{\big(\BM({\bf G}, \{h\})^{\bf K}\big)\mid {\bf K}\subset {\bf G}(\BA_f)\}$ defined over  the global Shimura field ${\bf E}={\bf E}({\bf G}, \{\mu_h\})$. Let $G={\bf G}\otimes_{\BQ}\BQ_p$. We fix an embedding $\bar\BQ\to\bar\BQ_p$.  Then we also obtain a conjugacy class of cocharacters $\{\mu\}$ of $G$ over $\bar\BQ_p$. Let  ${\bf p}$ be the  prime ideal of ${\bf E}$ over $p$ defined by the embedding of $\bar \BQ$ into $\bar\BQ_p$. Then $E={\bf E}_{\bf p}$ is the local Shimura field defined by $(G, \{\mu\})$. We take the open compact subgroups $\bf K$ of the form ${\bf K}={\bf K}^pK$, where $K\subset G(\BQ_p)$ is an open compact subgroup. Let 
\begin{equation}
\BM({\bf G}, \{h\})^{\bf K}(\CI)
\end{equation} be the tube over a basic isogeny class $\CI$  of the reduction modulo ${\bf p}$, an admissible open rigid-analytic subset of $\big(\BM({\bf G}, \{h\})^{\bf K}\otimes_{\bf E}\breve{E}\big)^{\rm rig}$.  It is not clear how to make sense of this term in this generality when $\BM({\bf G}, \{h\})^{\bf K}$ does not necessarily represent a moduli problem. We are assuming that the element $[b]\in B(G, \{\mu\})$ defined by $\CI$ is the unique basic element. Then there is an inner form $\bar{\bf G}$ of ${\bf G}$ which is anisotropic modulo center at the archimedean place, is isomorphic to $J_b$ (for $b\in [b]$) at the $p$-adic place and is isomorphic to ${\bf G}$ at all other places, and a system of  isomorphisms compatible with changes in ${\bf K}^p$ and $K$, with the action of ${\bf G}(\BA_f)$ by Hecke correspondences, and with descent data to $E$ on both sides, 
\begin{equation}
\bar{\bf G}(\BQ)\backslash \big(\BM(G, [b], \{\mu\})^K\times{\bf G}(\BA_f^p)/{\bf K}^p\big)\simeq \BM({\bf G}, \{h\})^{\bf K}(\CI). 
\end{equation}

There should be a variant  of this statement for non-basic isogeny classes $\CI$, cf. \cite{RZ}, Thm.~6.36. In this variant, $\bar{\bf G}$ is replaced by a group $I$ over $\BQ$ with $I\otimes_\BQ\BQ_p\simeq J_b$; the tube $\BM({\bf G}, \{ h\})^{\bf K}(\CI)$ is no longer an admissible open rigid-analytic subset of $\big(\BM({\bf G}, \{ h\})^{\bf K}\otimes_{\bf E}\breve E\big)^{\rm rig}$.

\end{remark}
  
\subsection{Examples}\label{secexlsh}Here are some examples of local Shimura varieties that are not RZ spaces.

\smallskip

 (i) The most basic case is the one where $G$ is a torus. Let $G=T$ be a torus over $F$ and $\mu\in X_*(T)=\Hom(\mathbb{G}_m, T_{\bar{F}})$ (in this case, the conjugacy class of $\mu$ consists of a single element).  Note that $B(G,\mu)$ consists of a single element, and that the set $A(G,\mu)$ is in bijection with $(X_*(T)_{\Gamma})_{\tor}$. Let $[b]\in B(G,\mu)$ and fix a representative $b$ of $[b]$. As $G=T$ is commutative, we have $J_b=G$. 

For tori, the local Shimura varieties will be towers of rigid-analytic spaces of dimension 0 and \'etale over $\Sp(\breve E)$. In order to describe them,  we use the following equivalence of categories. We denote by $I_E$ the inertia subgroup. 

Let $\overline x:\Sp(\hat{\overline{\mathbb Q}}_p)\rightarrow \Sp(\breve E)$ be a geometric point. Then, identifying  
$I_E$ with $\pi_1(\Sp(\breve E),\overline x)$, the fiber functor mapping each rigid space $Y$ to $Y_{\overline x}$ induces an equivalence of categories 
\begin{equation*}
\begin{aligned}
\big\{\text{rigid spaces  \'etale over $\Sp(\breve E)$}\big\}&\to \\ \big\{\text{discrete sets with a } &\text{continuous action of $I_E$}\big\} .
\end{aligned}
\end{equation*}

Using this equivalence of categories, we define for an  compact open subgroup $K$ of $T(\mathbb{Q}_p)$ 
\begin{equation}
\BM(T,[b], \mu)^K=T(\mathbb{Q}_p)/K ,
\end{equation}
with the following Galois action, analogous to Deligne's definition of the reciprocity law for special points on Shimura varieties, cf. \cite{Del-Bourb}, (3.9.1).

The cocharacter $\mu$ defines a homomorphism of tori 
$$
N_{\mu}:\Res_{E/F}\mathbb{G}_m\rightarrow T ,
$$
 which is defined as the composition 
 $$
 \Res_{E/F}\mathbb{G}_m\overset{\Res_{E/F}\mu}{\rightarrow}\Res_{E/F}T_E\overset{N_{E/F}}{\rightarrow} T ,
 $$ 
where $N_{E/F}$ denotes the norm map. Let 
$${\rm rec}_{E}:E^{\times}\rightarrow \Gal(\bar E/E)^{{\rm ab}}$$ 
be the Artin reciprocity map normalized in such a way that a uniformizer of $E$ is mapped to an arithmetic Frobenius in $\Gal(\bar{ E}/ E)$. Let 
\begin{equation}
\chi_{\mu}:I_E\rightarrow T(F)
\end{equation}
 be the composite
$$
I_E\rightarrow \mathcal{O}_E^{\times}\overset{N_{\mu}}{\rightarrow} T(F),\quad\gamma\mapsto N_{\mu}({\rm rec}_E^{-1}(\gamma_{|E^{{\rm ab}}})).
$$
 Then the action of $ I_E$ on $T(F)/K$ is given by 
 $$
 \gamma(xK)=\chi_{\mu}(\gamma)xK.
 $$
  For varying $K$, these spaces form a tower of rigid-analytic spaces over $\Sp(\breve E)$. The action of $J(F)\times T(F)$ is given by $(a,b)\cdot xK=abxK$. Note that,  as $T$ is abelian, the action maps each member of the tower  to itself. 
  
  Extending in the obvious way the action of $I_E$ to an action of $\Gal(\bar E/E)$, we see that  the whole tower is  defined over $E$,  so the Weil descent datum is effective in this case.

 Local Shimura varieties associated to tori are defined and used by  Chen \cite{C} in her study of geometrically connected components of RZ-spaces, cf. Conjecture \ref{conjconncomp}.
\smallskip

(ii) Another, in fact rather trivial, way in which local Shimura varieties arise is by taking a product of RZ spaces. Let $\CD_{\BZ_p, 1},\ldots, \CD_{\BZ_p, r}$ be an $r$-tuple of integral RZ data, with associated triples $(G_i, [b_i], \{\mu_i\})$ and associated formal schemes $\CM_i$ over $\Spf\, \CO_{\breve E_i}$. Let $(G, [b], \{\mu\})$ be the product of the $(G_i, [b_i], \{\mu_i\})$ in the obvious sense. Then the reflex field $E=E_{\{\mu\}}$ is the composite of the $E_i=E_{\{\mu_i\}}$, and we obtain a formal scheme $\CM$ over $\Spf\, \CO_{\breve E}$ as the product of the $\CM_i\hat\otimes_{\CO_{\breve E_i}}\CO_{\breve E}$. For a compact open subgroup $K$ of $G(\BQ_p)$ which is a product $K=K_1\times\ldots\times K_r$, we define the finite etale covering of $\CM^{\rm rig}$ as  
\begin{equation*}
\BM^K=(\BM_1^{K_1}\times_{{\rm Sp}\, \breve E_1}{\rm Sp}\, \breve E)\times\ldots\times(\BM_r^{K_r}\times_{{\rm Sp}\, \breve E_r}{\rm Sp}\, \breve E) ,
\end{equation*}
comp. Properties \ref{proplocSh}, (v). 
Since open compact subgroups of product form are cofinal among all open compact subgroups, we may define in the obvious way the tower $\{\BM^K\mid K\subset G(\BQ_p)\}$. It is the local Shimura variety associated to $(G, [b], \{\mu\})$.

\smallskip

(iii) Let  $(G,[b],\{\mu\})$  be a local Shimura datum, and let $H$ be a subgroup of $G$. Assume that $(H, [b]_H, \{\mu\}_H)$ is a local Shimura datum compatible with $(G,[b],\{\mu\})$ in the sense of Properties \ref{proplocSh}, (iii). Assuming that we already know how to define a local Shimura variety for $(G,[b],\{\mu\})$,  it would be desirable to define the local Shimura variety for the subgroup as a subobject.

One instance of this is the situation where $(G,[b],\{\mu\})$ corresponds to an RZ-datum $\CD$, and where $H$ is a Levi subgroup of $G$.
Let $V$ be given by the RZ datum and let 
\begin{equation}\label{dirsum}
V=V_1\oplus V_2\oplus\dotsm\oplus V_{r}
\end{equation}
be the decomposition of $V$ whose stabilizer is $H$. This decomposition is stable under the action of $B$.  

Let us first consider the EL case. Let $H_i$ be the group of $B$-linear automorphisms of $V_i$. Then $H_i$ is a subgroup of $H$, and $H\cong H_1\times\dotsm\times H_{r}$. If $\{\mu\}_H$ corresponds to an $r$-tuple of $H_i$-conjugacy classes $\{\mu_i\}$ of cocharacters, and $[b]_H$ comes from an $r$-tuple of classes $[b_i]$, then $(F, B, V_i, \{\mu_i\}, [b_i])$ are simple rational RZ data, with associated groups $H_i$. As simple integral RZ datum we take an $\CO_B$-stable lattice $\Lambda$ that decomposes as $\Lambda=\bigoplus\nolimits_i\Lambda_i$, where $\Lambda_i=V_i\cap\Lambda$, and obtain in this way also integral RZ data $\CD_{\BZ_p, i}$. If we take the framing object $(\BX, \iota_\BX)$ in product form, we obtain formal schemes 
$\CM_{\CD_{\BZ_p}}$ and $\CM_{\CD_{\BZ_p, i}}$ for $i=1,\ldots, r$. Let $E'=E(H, \{\mu\}_H)$. Then $E'$ is the composite of $E_1, \ldots,E_r$ and $E=E(G, \{\mu\})$ is contained in $ E'$. Let $\CO'=\CO_{\breve E'}$, $\CO=\CO_{\breve E}$ and $\CO_i=\CO_{\breve E_i}$. We obtain an obvious closed immersion of formal schemes 
\begin{equation}\label{closedemb}
(\CM_{\CD_{\BZ_p, 1}}\hat\otimes_{\CO_1}\CO')\times\ldots\times(\CM_{\CD_{\BZ_p, r}}\hat\otimes_{\CO_r}\CO')\to \CM_{\CD_{\BZ_p}}\hat\otimes_{\CO}\CO'.
\end{equation}
In terms of the moduli description of $\CM_{\CD_{\BZ_p}} $ by triples $(X, \iota, \rho)$, this closed formal subscheme is characterized as the subfunctor where the $p$-divisible group $X$, with its additional structure,  decomposes as a product. Thus the generic fiber of the source of (\ref{closedemb}) defines a subtower that can be regarded as the local Shimura variety associated to $(H, [b]_H, \{\mu\}_H)$. Of course, this is just an `internal version'  inside the local Shimura variety for $(G, [b], \{\mu\})$ of the product construction of (ii) above. 

Now consider the PEL case. We can choose the numbering in \eqref{dirsum} in such a way that 
$$
V_i\perp V_j \text{ \it unless $i+j=r+1$} . 
$$
In particular,  $V_{r+1-i}$ is the dual of $V_i$  with respect to the  pairing $(\,,\,)$. 

Let first $r=2r'$ be even.  For $1\leq i\leq r'$ let $H_i$ be the group of $B$-linear automorphisms of $V_i$. Then $$H\cong H_1\times\dotsm\times H_{r'}\times \mathbb{G}_m,$$ where an explicit isomorphism is given by mapping $(g_1,\dotsc,g_{r'},c)$ to the $B$-linear automorphism of $V$ whose restriction to $V_i$ (for every $i\leq r'$) is $g_i$ and which respects the pairing $(\,,\,)$ up to the scalar $c$. 

Now let $r=2r'-1$ be odd. In this case, there is an induced pairing on $V_{r'}$. For $1\leq i< r'$ let $H_i$ be the group of $B$-linear automorphisms of $V_i$, and let $H_{r'}$ be the group of $B$-linear automorphisms of $V_{r'}$ respecting the induced pairing up to a scalar. In this case $$H\cong H_1\times\dotsm\times H_{r'}$$ where an explicit isomorphism is given by mapping $(g_1,\dotsc,g_{r'})$ to the $B$-linear automorphism of $V$ whose restriction to $V_i$ (for every $i\leq r'$) is $g_i$ and which respects the pairing $(\,,\,)$ on $V$ up to the similitude factor $c(g_{r'})$ of $g_{r'}$ on $V_{r'}$.

In both the even and the odd case, we have $r'=\lfloor (r+1)/2\rfloor$. Let $\{\mu\}_H$ correspond to an $r'$-tuple of $H_i$-conjugacy classes $\{\mu_i\}$ of cocharacters, together with a conjugacy class $\{c\}$ of cocharacters of $\mathbb{G}_m$ if $r$ is even. Then the latter conjugacy class is a singleton and, by (\ref{eqcmu}), equal to the identity. Similarly $[b]_H$ comes from an $r'$-tuple of classes $[b_i]$, together with the class $[p]\in B(\mathbb{G}_m)$ if $r$ is even. Then for $i\leq r'$ if $r$ is even, and for $i<r'$ if $r'$ is odd, $(F, B, V_i, \{\mu_i\}, [b_i])$ are simple rational RZ data of EL type, with associated group $H_i$. For $r$ odd, 
$(F, B, V_{r'}, (\,,\,)|_{V_{r'}}, \ast, \{\mu_{r'}\}, [b_{r'}])$ is a simple rational RZ datum of PEL type with associated group $H_{r'}$. If $r$ is even, we choose $\CD_{0}=( \mathbb{Q}_p,  \mathbb{Q}_p, \mathbb{Q}_p, \{id\}, [p])$ as a simple rational RZ datum of EL type for the group $\mathbb{G}_m$. As simple integral RZ datum we take an $\CO_B$-stable lattice $\Lambda$ which is self-dual and which decomposes as $\Lambda=\bigoplus\nolimits_i\Lambda_i$, where $\Lambda_i=V_i\cap\Lambda$, and obtain in this way also integral RZ data $\CD_{\BZ_p, i}$, for $i=1,\ldots,r'$. We take the framing object $(\BX, \iota_\BX, \lambda_{\BX})$ in product form such that $\BX_i$ and $\BX_{r+1-i}$ are identified with the dual of each other by $\lambda_{\BX}$ for each $i$. Then we obtain formal schemes 
$\CM_{\CD_{\BZ_p}}$ and $\CM_{\CD_{\BZ_p, i}}$ for $i=1,\ldots, r'$. The formal scheme for $\CD_0$ is easy, and the corresponding rigid-analytic space  can be described explicitly as a special case of example (i) above, and is a discrete set isomorphic to $\mathbb{Z}$.

Let again $E'=E(H, \{\mu\}_H)$. Then $E'$ is the composite of $E_1, \ldots,E_{r'}$ and $E=E(G, \{\mu\})\subseteq E'$. As before, let $\CO'=\CO_{\breve E'}$, $\CO=\CO_{\breve E}$ and $\CO_i=\CO_{\breve E_i}$. We obtain closed immersions of formal schemes (the first one when $r$ is even, the second one when $r$ is odd), 
\begin{equation*}\label{closedembpel}
\begin{aligned}
(\CM_{\CD_{\BZ_p, 0}}\hat{\otimes}_{\CO_{\breve {\mathbb{Q}}_p}}\CO')\times(\CM_{\CD_{\BZ_p, 1}}\hat\otimes_{\CO_1}\CO')\times\dotsm\times(\CM_{\CD_{\BZ_p, r'}}\hat\otimes_{\CO_{r'}}\CO')\hookrightarrow \CM_{\CD_{\BZ_p}}\hat\otimes_{\CO}\CO'\\
(\CM_{\CD_{\BZ_p, 1}}\hat\otimes_{\CO_1}\CO')\times\dotsm\times(\CM_{\CD_{\BZ_p, r'}}\hat\otimes_{\CO_r'}\CO')\hookrightarrow \CM_{\CD_{\BZ_p}}\hat\otimes_{\CO}\CO'&.
\end{aligned}
\end{equation*}
In terms of the moduli description of $\CM_{\CD_{\BZ_p}}$ by tuples $(X, \iota,\lambda, \rho)$ the  maps are given as follows. When $r$ is even, an $r'+1$-tuple $((X_i,\iota_i,\rho_i))$ for $0\leq i\leq r'$ is mapped to $(X,\iota,\lambda,\rho)$ where $X=X_1\times\dotsm \times X_{r'}\times X_{r'}^{\vee}\times\dotsm\times X_1^{\vee}$, where $\iota$ is the product of the $\iota_i$ and their duals (for $i\geq 1$), where $\rho|_{X_i}=\rho_i$ and where $\rho|_{X_i^{\vee}}$ is the dual of $\rho_i$ up to the factor $p^{c_0}$ where $c_0\in \mathbb{Z}$ is the height of $\rho_0$. The polarization $\lambda$ identifies the dual of the $i$th factor of $X$ with the $2r'+1-i$th factor for each $i$. When $r$ is odd, the map is defined similarly, except for the fact that the similitude factor is determined by the $r'$th component of the left hand side, which contains a polarization. In terms of the tuples $(X, \iota,\lambda, \rho)$, the closed formal subscheme obtained as the image of either of these immersions is characterized as the subfunctor where everything decomposes as a product. Thus the generic fiber of the source of these maps defines a subtower of the local Shimura variety associated with $(G, [b], \{\mu\})$ that can be regarded as the local Shimura variety associated to $(H, [b]_H, \{\mu\}_H)$. As in the EL case this can also be seen as a variant of the product construction of (ii) above.

\smallskip

(iv)  The following construction was pointed out by Fargues. Let $(G, [b], \{\mu\})$ be a local Shimura datum over $F$, and let $c:\BG_{m, F}\to G$ be a central cocharacter defined over $F$. Then we obtain a new local Shimura datum $(G' , [b'], \{\mu'\})$ over $F$ by setting $G'=G$, and $b'=bc(\pi)$  (where $\pi$ is a uniformizer in $F$), and $\mu'=\mu c$. If $\BM$ is a local Shimura variety for  $(G, [b], \{\mu\})$, then one should be able to  construct a local Shimura variety $\BM'$ for  $(G', [b'], \{\mu'\})$ by twisting $\BM$ by the cyclotomic character. For instance, if  $(G, [b], \{\mu\})$ comes from simple integral RZ-data of type EL with $B=F$, and if $c(z)=z^d{\rm Id}_V$, then instead of considering  level structures $\eta: (\BZ/p^k)^n\simeq X[p^k]$ on the universal $p$-divisible group $X$ in the generic fiber, one considers level structures $\eta: (\BZ/p^k)^n(d)\simeq X[p^k]$. The resulting local Shimura variety is not of RZ type for $d\neq 0$.

\smallskip

(v) Let $G$ be an arbitrary reductive group over a local field $F$. Let $\{\mu\}=\mu_0$ be the trivial cocharacter of $G$ and $[b]\in B(G, \mu_0)$. Then $J_b = G$ for $b\in [b]$.
In this case, the local Shimura variety should just be the  discrete set
\begin{equation*}
 \BM^{K} = G (F) / K \, ,
\end{equation*}
with $J(F)=J_b(F)=G(F)$ acting by left translations, and $G(F)$ acting on the tower by right translations. 

\section{$\ell$-adic cohomology of local Shimura varieties}\label{elladiccoho}

In this section we explain what we mean by the $\ell$-adic cohomology of local Shimura varieties.
We fix a prime number $\ell \neq p$.

We fix a local Shimura datum $(G, [b], {\{ \mu \}})$ giving rise to the local Shimura variety 
\begin{equation}
\BM(G, [b], \{\mu\})=\{ \BM^{K} \mid K \subset
G (F) \} .
\end{equation}
 We will use $\ell$-adic cohomology with compact supports, denoted for each $i \in \BN$ by
\begin{equation*}
 H^{i}_c (\BM^{K}) = H^{i}_c (\BM^{K} \times_{\Sp\, \breve{E}} \Sp\, \hat{\bar{E}}, \bar{\BQ}_{\ell} ) \, ,
\end{equation*}
 where $\hat{\bar{E}}$ denotes the completion of an algebraic closure of $E$.
Here we view $\BM^{K}$ either as an adic space or as a Berkovich space, and we use the cohomology theory
developed in either context. Since the theory for Berkovich spaces is better documented, we often use the latter.
The $\bar{\BQ}_{\ell}$-vector space $H^{i}_c (\BM^{K})$ is equipped with an action of $J(F)$, by functoriality. It is also
equipped with an action of the inertia group $I_E = \Gal (\hat{\bar{E}} / \breve{E})$ by functoriality. Due 
to the Weil descent datum from $\breve{E}$ to $E$, the action of $I_E$ extends to an action of the Weil group
$W_E$. Furthermore, $$H^{i}_c (\BM^{K}) = 0 ,\text{ for } i > 2 \dim \BM^{K} ,$$ cf. \cite{Fargues}.
\begin{proposition}\label{cohoprop}

 \noindent(i) For every $K$ and every $i$, the action of $J(F)$ on $H^{i}_c (\BM^{K})$ is smooth and the action
of $I_E$ is continuous, in the sense that $H^{i}_c (\BM^{K})$ is the increasing union of finite-dimensional
continuous representations of $I_E$. Furthermore, $H^{i}_c (\BM^{K})$ is a finitely generated $J (F)$-module.

\smallskip

\noindent (ii) For every admissible representation $\rho$ of $J(F)$, and every $K$ and $i, j$, the 
$\bar{\BQ}_{\ell}$-vector space 
\begin{equation*}
 H^{i, j} (\BM^{K}) [\rho] = {\rm Ext}^j_{J(F)} (H^{i}_c (\BM^{K}), \rho)
\end{equation*}
is finite-dimensional, and vanishes for $j > {\rm rk}_{\rm ss} (J)$. 

{\rm Here the ${\rm Ext}$ is taken in the abelian category of smooth $J(F)$-modules. }

\smallskip

\noindent (iii) For every admissible representation $\rho$ of $J(F)$ and every $i, j$, the module
\begin{equation*}
H^{i,j}((G, [b], {\{ \mu \}})) [\rho] = {\varinjlim}_{{K}} {\rm Ext}^{j}_{J(F)} 
(H^{i}_c (\BM^{K}), \rho)
\end{equation*}
 is an admissible $G(F)$-module and a continuous $W_E$-module.
\end{proposition}
\begin{proof}
 (i) The action of the analytic group $J (F)$ on $\BM^{K}$ is continuous, and this implies that 
$H^{i,j} (\BM^{K})$ is a smooth $J(F)$-module since $\BM^{K}$ is smooth and hence quasi-algebraic, cf. \cite{Fargues}, Cor. 4.1.19
(which uses results of Berkovich for the cohomology with torsion coefficients).
The assertion regarding the $I_E$-action is \cite{Fargues}, Cor. 4.1.20.

To prove that $H^{i}_c (\BM^{K})$ is a finitely generated $J (F)$-module, we make use of Properties \ref{proplocSh}, (iii) of local Shimura varieties (recall that this property is not known to hold in all cases of RZ spaces,  see Remark \ref{relRZlocsh}, and also Remarks \ref{remsoncoh}, (i) below.) Let $U$ be an open subset of $\BM^{K}$ as given by that property. Then by \cite{Hub-comp}, Thm.~3.3, (ii), $H^{j}_c (U)$ is finite-dimensional, for all $j$.  Imitating the argument in \cite{Fargues}, proof of Prop. 4.4.13., one constructs a spectral sequence that converges to the cohomology $H^*(\BM^{K})$, with only finitely many terms on each anti-diagonal and such that each term is a finite sum of  representations which are compactly induced from finite-dimensional representations of compact open subgroups of $J(F)$. The result follows.

(ii) This is a general fact about the category of smooth, resp. admissible, representations of $J(F)$,  
cf. \cite{Fargues}, Cor. 4.4.14 and Lemma 4.4.12.

(iii) is a consequence of (i) and (ii).
\end{proof}
\begin{remarks}\label{remsoncoh}
 (i) In the case of a local Shimura variety coming from simple unramified RZ data (cf. Remark \ref{unramifiedRZ}),
Fargues has shown (\cite{Fargues}, Prop. 4.4.13) that $H^i_c (\BM^{K})$ is a $J(F)$-module of finite
type, by establishing Property \ref{proplocSh}, (iii) in this case. His proof uses in an 
essential way the structure of the formal schemes attached to RZ data.  More general cases are due to Mieda \cite{Mieda-ram}, again based on structure results on the formal schemes attached to RZ data. 

Note that  $H^i_c (\BM^{K})$ is not of finite length in general, cf. \cite{Fargues}, Rem. 4.4.18. One cause of this is if the center of $J(F)$ is not compact. However, most of the time, even the cohomology $H^i_c (\BM^{(\delta), K})$ of $\BM^{(\delta), K}=(\kappa^K)^{-1}(\delta)$ is not of finite length over $J(F)^1$ (see \eqref{periodmorph} for the notation $\kappa^K$). Here $\delta\in \Delta$ and $J(F)^1$ is  the intersection of all kernels of $F$-rational characters of $J$. However, this finite length property holds in the Lubin-Tate case, and hence (by using the duality  with the Drinfeld tower) also in the Drinfeld case, comp.~\cite{Dat-ell}, Cor.~3.5.11. 

(ii) In \cite{Mieda-gsp4}, Def.~3.3, Mieda introduces $H^{i,j}((G, [b], {\{ \mu \}})) [\rho]$  by taking the smooth vectors in the
$G(F)$-module
\begin{equation*}
 {\rm Ext}^j_{J(F)} ({\varinjlim}_{{K}} H_c^i(\BM^{K}), \rho) \, .
\end{equation*}
By \cite{Mieda-gsp4}, Remark 3.4, this gives the same result as above.

(iii) It seems likely that the $G(F)$-module $H^{i,j}((G, [b], {\{ \mu \}})) [\rho]$ is of finite length if $\rho$ is of finite length (since this
$G(F)$-module is admissible, it is equivalent to ask whether this $G(F)$-module is finitely generated (Howe's theorem), cf.~\cite{Renard}, VI.6.3). We
refer to \cite{Mieda-gsp4}, Remark 7.10 for ideas in this direction.

(iv) As pointed out in \cite{Fargues}, Rem. 4.4.4., a Frobenius element in $W_E$ will not in general preserve
a finite-dimensional subspace of $H^i_c (\BM^{K})$. This shows that the $W_E$-action does not extend to a
continuous action of $\Gal (\bar{E} / E)$.
\end{remarks}
We denote by ${\rm Groth }(G(F) \times W_E)$ the Grothendieck group of the category of 
$\big(G(F) \times W_E\big)$-modules over $\bar{\BQ}_{\ell}$ which are admissible as $G(F)$-modules and continuous as $W_E$-modules. Our
object of interest will be, for a fixed irreducible representation $\rho$ of $J(F)$,  the alternating sum, modified by a Galois twist
\begin{equation}
 H^{\bullet} ((G, [b], {\{ \mu \}})) [\rho]= \sum\nolimits_{i, j \geq 0} (-1)^{i + j} H^{i,j}((G, [b], {\{ \mu \}})) [\rho] (-\dim \BM) ,
\end{equation}
considered as an element of $ {\rm Groth} (G(F) \times W_E)$. This makes sense since only finitely many terms in this sum are non-zero.

\section{Realization of supercuspidal representations}\label{seckottwitzconj}
In this section we state the conjecture of Kottwitz \cite{Rapo-ICM}, Conj. 5.1,  in the context of local
Shimura varieties and adapted to the set-up of the last section. We fix a local Shimura datum $(G, [b],\{ \mu \})$
giving rise to a local Shimura variety, where we assume that $[b]$ is \emph{basic}.

\subsection{The Kottwitz conjecture} To state the conjecture, we will have to make use of some aspects of a conjectural Langlands correspondence.
A Langlands parameter (cf.~\cite{GGP}, \S 8) for a group $G$ over $F$,
\begin{equation*}
 \varphi : W_F \lra ^LG = \hat{G} \rtimes W_F \, ,
\end{equation*}
is called \emph{discrete}, if the connected component $S^{\circ}_{\varphi}$ of the centralizer group 
$S_{\varphi} = \rm{Cent}_{\hat{G}} (\varphi)$ lies in $Z(\hat{G})^{\Gamma}$.  
\begin{remark}
We point out that we are using the Weil group $W_F$, and not the Deligne-Weil group, i.e., there is no $\SL_2$-factor present here. In \cite{Kottwitztemp}, such Langlands parameters are called {\it elliptic}. By \cite{Kottwitztemp}, Lemma 10.3.1, a Langlands parameter is discrete if and only if the image of $\varphi$ is not contained in any proper Levi subgroup  of $^LG$. Note that if $G = \GL_n$, then $L$-packets are singletons, and a representation $\pi$ corresponds to a discrete Langlands parameter if and only if $\pi$ is supercuspidal. 
If $G={\rm U}_3$ or $G = {\rm GSp}_4$, in which cases a Langlands correspondence has been constructed, the discrete $L$-packets  are exactly those that consist entirely of supercuspidal representations ( for the case of ${\rm U}_3$, cf.~\cite{Fargues}, App.~C.; for the case of ${\rm GSp}_4$, cf.~\cite{Mieda-gsp4}, Thm. 7.3 for one direction\footnote{According to W.~T.~Gan (e-mail to the authors), both directions follow easily from the construction in \cite{Gan-Takeda}. }).
\end{remark}
 To state the conjecture, we make the
assumption that $G$ is a {\it $B$-inner twist} of the quasi-split form $G^*$ of $G$, i.~e., there exists a basic
element $b^* \in G^* (\breve{F})$ such that $G$ is isomorphic to the inner twist $J^*_{b^*}$ of $G^*$
defined by $b^*$, cf.~section \ref{sigmaconjclas}. 
\begin{remark}
Every reductive group $G$ can be written in this way, provided that the center of $G$ is connected. It would be interesting to find a formulation of the Kottwitz conjecture to cover the most general case.  For this,  one might have to employ Kaletha's  extended cohomology groups $H^1(u\to W_F, Z\to G)$, cf.~\cite{Kal-rig}.
\end{remark}
 Under the canonical identifications
\begin{equation*}
 \pi_1 (G)_{\Gamma} = X^* \big(Z(\hat{G})^{\Gamma}\big),\quad \pi_1 (G^*)_{\Gamma} = X^*\big(Z(\hat{G})^{\Gamma}\big)\, ,
\end{equation*}
we obtain elements of $X^* \big(Z(\hat{G})^{\Gamma}\big)$,
\begin{equation*}
 \lambda_b = \kappa_G ([b]), \quad \lambda_{b^*} = \kappa_{G^*} ([b^*])\, .
\end{equation*}
According to the refined Langlands correspondence (cf. \cite{Kal-rig, Vogan, GGP}), there should be identifications $\pi\mapsto \tau_\pi$, resp. $\rho\mapsto \tau_\rho$, 
of the $L$-packets for $G$, resp. for $J$, 
\begin{equation*}
\begin{aligned}
 \Pi_{\varphi} (G) & = \{ \text{irreducible  algebraic rep'ns $\tau$  of
$S_{\varphi}$} \mid \tau \vert Z (\hat{G})^{\Gamma} = \lambda_{b^*} \}\\
 \Pi_{\varphi} (J) & = \{ \text{irreducible  algebraic rep'ns $\tau$  of
$S_{\varphi}$} \mid \tau \vert Z (\hat{G})^{\Gamma} = \lambda_{b^*}+\lambda_b \} 
\end{aligned}
\end{equation*}
where the equalities are between elements of $X^*(Z(G)^\Gamma)$. Of course, we do not know how to characterize these bijections. However, in the cases  where these bijections are known, they can be pinned down by imposing endoscopic character relations with respect to a choice of transfer factors that depend on the choice $\frak w=(B, \psi)$ of Whittaker data, cf. \cite{Kal-gen}, \S 1.  In particular,  $\tau_\pi$ is the trivial one-dimensional representation of $S_\varphi$ if and only if $\pi$ is the $(B, \psi)$-generic representation of $G^*(F)$. Any other choice $\frak w'$ of Whittaker datum differs from $\frak w$ by a one-dimensional character $\lambda=\lambda(\frak w, \frak w')$ of $S_\varphi$ that is trivial when restricted to $Z (\hat{G})^{\Gamma}$, cf. \cite{Kal-gen}, \S 1. When $\frak w$ is replaced by $\frak w'$, then $\tau_{\pi}$ and $\tau_\rho$ are replaced by $\tau_{\pi}\otimes\lambda$ and $\tau_\rho\otimes \lambda$. Hence the function
\begin{equation*}
\begin{aligned}
 \Pi_{\varphi} (G) \times \Pi_{\varphi} (J) & \lra {\rm{Rep}} (S_{\varphi})\\
(\pi \, , \rho) & \longmapsto \check{\tau}_{\pi} \otimes \tau_{\rho}
\end{aligned}
\end{equation*}
is well-defined, i.e., independent of the choice of Whittaker data. Here $\check{\tau}_{\pi}$ denotes the contragredient representation of 
$\tau_{\pi}$.  We refer to \cite{Kal-gen}, Thm. 4.3  for cases in which this general idea is realized in specific cases (related  to symplectic and special orthogonal groups). 

To state the conjecture, we recall the finite-dimensional representation $r_{\{\mu\}}$ of $\hat{G} \rtimes W_E$
defined by $\{\mu\}$, cf. \cite{Kottwitz0}, Lemma~2.1.2, which we regard here as a representation in a $\bar{\BQ}_{\ell}$-vector space.
\begin{conjecture}[Kottwitz conjecture]\label{kottwitzconj}
 Under the assumptions on $(G, [b], \{ \mu \})$ above, let $\varphi$ be a discrete
Langlands parameter for $G$. Denoting by $\varphi_E$ the restriction of $\varphi$ to $W_E$, regard
$r_{\{ \mu \}} \circ \varphi_E$ as a representation of $S_{\varphi} \times W_E$, via
\begin{equation*}
 \big(r_{\{ \mu \}} \circ \varphi_E \big)(s, w) = r_{\{ \mu \}} (s \cdot \varphi_E (w) ) \, .
\end{equation*}
Then, for $\rho \in \Pi_{\varphi} (J)$,
\begin{equation*}
 H^{\bullet} ((G, [b], \{ \mu \})) [\rho] = \sum_{\pi \in \Pi_{\varphi}(G)} {\pi} \boxtimes
\Hom_{S_{\varphi}} (\check{\tau}_{\pi} \otimes \tau_{\rho} , r_{\{ \mu \}} 
\circ \varphi_E) (-\frac{d}{2}) \, .
\end{equation*}
\end{conjecture}
Here the second factor on the RHS is a $W_E$-module in the obvious way, that has been twisted by minus half the
dimension of $\BM$.

\begin{remarks}\label{remkottwitzconj}
 (i) Since we are assuming in this conjecture that $\varphi$ is a  discrete Langlands parameter, the LHS should
simplify since one can expect in this case that for every $i$
\begin{equation*}
 {\rm Ext}^j_{J (F)} ( H^i_c (\BM^{K'}), \rho ) = 0 , \, j > 0 \, .
\end{equation*}
This indeed holds if $\rho$ is supercuspidal, cf. \cite{Renard}, VI.3.6.

\smallskip

 (ii) In many cases one has
\begin{equation*}
 \Hom_{J(F)} (H^i_c (\BM^{K}), \rho ) = 0 ,  \,  i \neq d \, ,
\end{equation*}
e.~g.~\cite{Dat-Dr}, although this probably cannot be expected in general. But this vanishing
property dictates the sign factor on the RHS.

\smallskip

(iii) We note that if $G$ is quasisplit,  the formula above simplifies.
Indeed, in this case $G = G^*$ and $\lambda_{b^*}  = 0$.

\smallskip

(iv) A weakened version of Conjecture \ref{kottwitzconj} would be to postulate the above formula after taking the isotypic part of the LHS corresponding to representations $\pi$ which lie in some (unspecified) discrete $L$-packet. 

\smallskip

(v) A more symmetric form  of Conjecture \ref{kottwitzconj}, closer to the original version in \cite{Rapo-ICM}, Conj. 5.1\footnote{Note that in loc.~cit. the contragredient sign for $\pi$ should be eliminated. This error has been also copied into \cite{H}, Conj.~5.3.}, and which relates to the duality conjecture \ref{dualityconj}, is due to Scholze \cite{Scholze-surv}, \S 5. Let $C=\hat{\bar E}$. As in Conjecture \ref{dualityconj}, we use the notation $(G, [b], \{\mu\})$ and $(G^\vee, [b^\vee], \{\mu^\vee\})$, and we assume the existence of the perfectoid space $\CM$ over $C$, as in the remarks following Conjecture \ref{dualityconj}.  Then by \cite{Scholze-surv}, Prop.~5.4. (iii), there is an isomorphism of {\it smooth} $G(F)\times G^\vee (F)$-modules,
\begin{equation}\label{isolim}
{\varinjlim}_{{K}} H_c^i(\BM^{K})\simeq {\varinjlim}_{{K^\vee}} H_c^i((\BM^\vee)^{K^\vee}) .
\end{equation}
Let $\pi$ be an irreducible $G(F)$-module, and $\rho$ an irreducible $G^\vee (F)$-module. Let  $ H^{\bullet} ((G, [b], \{ \mu \})) [\pi\times \rho]$ be the $\pi\times \rho$-isotypical component of the alternating sum of the modules in \eqref{isolim} (defined in terms of Ext-groups of smooth $G(F)\times G^\vee (F)$-modules, analogous to Proposition \ref{cohoprop}, (iii)). Let us assume that this $W_E$-module is non-zero. Then the conjecture is that $\pi$ lies in a discrete $L$-packet if and only if $\rho$ does; and that in this case the $L$-parameters should be identical (up to equivalence); and that, in this case, denoting by $\varphi$ the corresponding $L$-parameter, 
\begin{equation*}
 H^{\bullet} ((G, [b], \{ \mu \})) [\pi\times \rho]\simeq \Hom_{S_{\varphi}} (\check{\tau}_{\pi} \otimes \tau_{\rho} , r_{\{ \mu \}}  \circ \varphi_E) (-\frac{d}{2}) .
\end{equation*}
Note that the RHS of this relation is symmetric in $\pi$ and $\rho$.

\end{remarks}
\begin{remark}
 It would be interesting to extend this conjecture to include non-discrete Langlands parameters. In the
Drinfeld case and the Lubin-Tate case this has been done by Dat \cite{Dat-ell}. However, more general cases
seem too speculative at the moment (they seem to involve {\it discrete Arthur parameters}). 
\end{remark}
\subsection{Known results} We review here some results concerning this conjecture.

\smallskip

\noindent a) As a quick check that the conjecture is reasonable, let us consider the case where ${\{\mu\}}= \mu_0$ is the
trivial cocharacter. Then, taking $b=1$ as representative of $[b]$, we have  $J_b = G$. 
In this case, the local Shimura variety is the discrete set
\begin{equation*}
 \BM^{K} = G (F) / K \, ,
\end{equation*}
with $J(F)=J_b(F)=G(F)$ acting by left translations, and $G(F)$ acting on the tower by right translations, cf. Examples \ref{secexlsh}, (v).
One checks easily that
\begin{equation}\label{kottwtriv}
 H^\bullet ((G, [b], {\{\mu\}}) [\rho] = {\Hom}_{J(F)} (C^{\infty}_c (G(F)), \rho ) .
\end{equation}
Now using the fact that the $L$-parameter of $\rho$ is discrete, it should follow that the RHS of
\eqref{kottwtriv} is equal to a multiple of $\rho$. If $\rho$ is a supercuspidal representation, it is in fact equal to $\rho$. Here our guide is the Peter-Weyl formula, valid when $G(F)$ is compact,
\begin{equation*}
 C^{\infty} (G(F)) = \bigoplus_{\rho \in \hat{G}(F)} \rho \otimes \check{\rho} \, .
\end{equation*}
Hence, assuming that the multiplicities check, we see that the conjecture is reasonable in this case.

\smallskip

\noindent b) The next  batch of results concerns cases when $G = {\rm Res}_{F/\BQ_p} (\GL_n)$. Note that in this case
Langlands parameters exist, and each $L$-packet $\Pi_{\varphi} (G)$ consists of a single element. Something
similar is true of $J_b$. Through the functoriality  (vi) of Properties \ref{proplocSh}, this can be reduced to the case $G=\GL_n/F$. 
\begin{theorem}[Harris/Taylor \cite{HT}, Thm. VII.1.3] Let $G = \GL_n/F$ and 
$${\{\mu\}} = (1,0, \ldots, 0)\in (\BZ^n)_{\geq }.
$$
 Then the Kottwitz conjecture is true for the triple $(G, [b], {\{\mu\}})$, where $[b] \in 
B(G,  {\{\mu\}})$ is the unique basic element.
\end{theorem}
Outside the case of the Drinfeld cocharacter, there is the following result. 
\begin{theorem}[Shin \cite{Shin}, Cor.~1.3]
 Let $G = {\rm Res}_{F/\BQ_p} (\GL_n)$, where $F/\BQ_p$ is an unramified extension. Then the Kottwitz conjecture
is true for $(G, [b], {\{\mu\}})$.
\end{theorem}
Shin's proof uses global methods, and also uses a point counting argument on Igusa varieties.   Earlier, Fargues  \cite{Fargues}, Thm.~8.1.5, had proved the weakened version of the Kottwitz conjecture of Remarks \ref{remkottwitzconj} (iv) in this case under the condition that $J_b$ is anisotropic modulo center.  We note that Fargues states in  \cite{Fargues}, Thm.~8.1.5 that the same conclusion  also holds if  $J_b\simeq G$. However, Fargues has informed us that Mieda pointed out a gap in the proof in loc.~cit., related to the use of the  trace formula. 

\smallskip

\noindent c) Outside these cases related to $\GL_n$, the conjecture does not seem to be proved for all $\rho$ for a
single group. For unitary groups, Fargues states in  \cite{Fargues}, Thm.~8.2.2, the Kottwitz conjecture  for $(G, [b], {\{\mu\}})$ in its weakened version of Remarks \ref{remkottwitzconj} (iv), under the following hypotheses. First, 
 $G = {\rm Res}_{F_0/\BQ_p} ({\rm GU}_3)$, where $F_0/\BQ_p$ is an unramified extension of odd degree, and where 
${\rm GU}_3$ is the quasi-split group of unitary similitudes in $3$ variables for an unramified quadratic
extension $F/F_0$.  Furthermore,   ${\{\mu\}}$ is arbitrary and $[b] \in  B (G, {\{\mu\}})$ is the
unique basic element. Finally, 
$\rho$ is a stable supercuspidal representation of $J_b (\BQ_p)$.
Note that then $J_b = G$ (as always for a basic element of an odd group of unitary similitudes). The
hypothesis that $\rho$ be a stable supercuspidal representation is equivalent to asking that the
$L$-packet of $\rho$ corresponds to a discrete $L$-parameter and consists of only one element,
cf. \cite{Fargues}, App. C.4. However, Fargues has informed us that the proof in loc.~cit. contains the same kind of gap as the one occurring in the proof of  \cite{Fargues}, Thm.~8.1.5, comp. b) above. 

Mieda has informed us of joint work in progress with T. Ito that would improve the situation. They consider the group $G={\rm GSp}_4$ and also unramified unitary groups of signature $(n-1, 1)$.

\begin{remark}
 A weakening of the Kottwitz conjecture would be to disregard the action of the Weil group $W_E$.
For this weakening one may hope to obtain a purely local proof. A general proof seems out of sight, but 
there are encouraging partial results which we enumerate as follows.
\end{remark}
\begin{altenumerate}
 \item There is the purely local paper by Strauch \cite{Strauch-adv} that considers the weakened conjecture in the Lubin-Tate
 case for an arbitrary $F/\BQ_p$. His method is an extension of Faltings' method of using the Lefschetz
fixed point formula for the local Shimura variety, \cite{Falt-Dr}.
\item There is a series of papers by Mieda \cite{Mieda-LT, Mieda-lJL, Mieda-gsp4} in which he transposes the method of Strauch to
the case when $G= {\rm GSp}_4$. He comes close to proving the weak version of the Kottwitz conjecture in this
case.
\end{altenumerate}

\section{Cohomology in the non-basic case} \label{harrisconj}

In the previous section, we considered local Shimura data $(G, [b], {\{ \mu \}})$ where $[b]$ is basic and
gave a conjectural formula for the element $H^{\bullet} ((G, [b], {\{ \mu \}}))[{\rho}]$  of $ {\rm Groth} (G(F) \times W_E)$, provided
that $\rho$ is an irreducible admissible representation of $J_b(F)$ with discrete $L$-parameter. In the
present section, we drop the assumption that $[b]$ is basic. In this case we are not able to give a precise
formula for $H^{\bullet} ((G, [b], {\{ \mu \}}))[{\rho}]$, not even if $\rho$ has a discrete $L$-parameter\footnote{In \cite{Shin} Shin gives in special cases an interesting {\it averaging formula} over $B(G, \{\mu\})$ for a fixed pair $(G, \{\mu\})$.   }. However,
Harris has proposed an inducing formula in this case, cf. \cite{H}. We will give a modified version of
Harris's formula\footnote{In the Zurich lectures of  the first author, this modified version was called
the \emph{Harris-Viehmann conjecture}.}.
To simplify the formula, we will make the  assumption that $G$ is quasi-split. 

At this point, we should stress that this conjecture has met with a lot of scepticism among people around us (and we ourselves are sympathetic to this group!). But, even if the conjecture turns out not to be true in its precise prediction, we believe that something in this direction should be valid, and that trying to prove it should teach us interesting things about local Shimura varieties. One simple reason why the conjecture might have to be modified is that it is not clear whether the definition of $\ell$-adic cohomology in section \ref{elladiccoho} is the correct one in this context\footnote{For the Kottwitz conjecture, the correct choice of the cohomology theory should not matter; however, for the HV conjecture it possibly could be important. This is analogous to the situation in the global case where, in order to accomodate the contributions of all automorphic representations to the cohomology of a Shimura variety, one uses  the middle intersection cohomology of the Baily-Borel compactification. It is conceivable that one has to define and use a similar cohomology theory also in the local case considered here. }\label{footcoho}.

\subsection{The HV conjecture} Let $L$ be a Levi subgroup of $G$ such that $b \in L (\breve{F})$ for some $b \in [b]$. Let 
\begin{equation}\label{defofI}
\begin{aligned}
 I_{b, {\{ \mu \}}, L} = \{ \mu' \mid \text{cocharacter of $ L $} \text{ conjugate to }\{ \mu \}\text{ in $ G$}, \\ \text{ with $ [b]_L \in B(L, \{ \mu' \}_L )$} \},
\end{aligned}
\end{equation}
taken modulo $L$-conjugacy.

\smallskip

Here are some facts about the sets $ I_{b, {\{ \mu \}}, L}$. 

\begin{lemma}
Let $L$ and $b\in L(\breve{F})$ be as above. Then 
\begin{altenumerate}
\item $ I_{b, {\{ \mu \}}, L}$ is finite.
\item If $G$ is unramified, then $ I_{b, {\{ \mu \}}, L}$ is non-empty. 
\item If $G$ is split,  then $I_{b, {\{ \mu \}}, L}$ consists of a single element.
\end{altenumerate}
\end{lemma}
\begin{proof}
Let $T$ be a maximal torus of $G$ contained in $L$. Then $\{\mu\}$ has a representative in $X_*(T)$, and similarly for each $\{\mu'\}_L\in I_{b, {\{ \mu \}}, L}$. By the first condition of \eqref{defofI} these representatives are in the same orbit under the action of the absolute Weyl group of $T$ in $G$. As this group is finite, the same holds for $I_{b, {\{ \mu \}}, L}$. 

For the non-emptiness statement in (ii) it is enough to consider Levi subgroups over $F$ which are minimal with the property that they contain an element of $[b]$. After replacing $L$ and the representative of $[b]$ by a conjugate we may assume that $L$ is in addition the Levi factor of a standard parabolic subgroup (with respect to a chosen Borel containing $T$) and that the $L$-dominant Newton point of the representative of $[b]$ is $G$-dominant. Then the statement follows from \cite{CKV}, Proposition 3.4.1. Indeed the set $\bar I_{\mu,b}$ considered in loc.~cit. coincides with our set $I_{b, {\{ \mu \}}, L}$, since $\{\mu\}$ is minuscule.

Now let us prove  the last statement. We use the fact that the projection  
\begin{equation}\label{minuscrep}
\{\mu'\in X_*(T)\mid \text{ $\mu'$ $L$-dominant and $L$-minuscule} \}\rightarrow \pi_1(L)
\end{equation}
is bijective (see \cite{Bou}, ch. VI, \S 2, ex. 5). Note that, as $G$ and hence $L$ are split, we have $\pi_1(L)_{\Gamma}=\pi_1(L)$. Elements $\{\mu'\}_L\in I_{b, {\{ \mu \}}, L}$ have a unique representative in the left hand set.  However, $[b]_L\in B(L,\{\mu'\}_L)$ implies that such a representative maps under \eqref{minuscrep} to $\kappa_L(b)$, considered as an element of $\pi_1(L)_{\Gamma}=\pi_1(L)$. Hence there is a unique such representative.  \end{proof}

\begin{remark}
We expect that the second assertion of the lemma also holds without the assumption that $G$ is unramified. However, the proof of \cite{CKV}, Prop. 3.4.1 does not directly generalize to that context.
\end{remark}

\begin{example}\label{exIbig} The following example shows that for non-split $G$ the 
set $ I_{b,{\{ \mu \}}, L}$ may have more than one element.  Note that this contradicts \cite{H}, Proposition 4.1 (ii); the false statement is also
used in the original formulation of Harris' conjecture \cite{H}, Conjecture
5.2.

Let $G=\Res_{F/\mathbb{Q}_p}\GL_5$ where $F/\mathbb{Q}_p$ is an unramified
extension of degree 2. Let $[b]$ be a $\sigma$-conjugacy class with Newton vector 
\begin{equation*}
\nu_{[b]}= \big((1/4)^{(4)}, (2/3)^{(6)}\big) . 
\end{equation*}
Let
$L=\Res_{F/\mathbb{Q}_p}(\GL_2\times \GL_3)$. We choose $b\in [b]\cap
L(\breve F)$ such that the image in the first factor of $L$ is basic of slope
$1/4$ and the image in the other factor is basic of slope $2/3$. Then $J_b$ is an inner
form of $L$.

We fix the Borel subgroup $B=\Res_{F/\mathbb{Q}_p}B_0$ where $B_0$ are the
lower triangular matrices and let $T=\Res_{F/\mathbb{Q}_p}T_0$ where $T_0$
is the diagonal torus in $\GL_5$. Then we can represent conjugacy classes
of cocharacters of $G$ by their unique representative in the $G$-dominant
elements of $X_*(T)$. Similarly we represent conjugacy classes of
cocharacters of $L$ by $L$-dominant elements of $X_*(T)$, where we choose
$B\cap L$ as Borel subgroup of $L$. Over $\breve F$ we can identify $G$
with a product of two factors $\GL_5$. Using this and the
identification $X_*(T_0)=\mathbb{Z}^5$ we can write dominant elements of
$X_*(T)$ as elements of $(\mathbb{Z}^5)^2$ such that each tuple
$(\mu_1,\dotsc,\mu_5) \in \mathbb{Z}^5$ satisfies $\mu_1\leq\dotsm\leq\mu_5$.

Let
\begin{equation*}
\mu=\big(\big(0^{(3)}, 1^{(2)}\big), \big(0^{(2)}, 1^{(3)}\big) \big) . 
\end{equation*}
Then  $[b]\in B(G,\{\mu\})$.
\smallskip

\noindent{\bf Claim.} {\it The set $I_{b, {\{ \mu \}}, L}$ contains two
elements $\mu_1, \mu_2$. The two local Shimura varieties  corresponding to $(L, [b]_L, \{\mu_i\}_L)$ for $i=1,2$ have different dimensions, and are in particular not isomorphic.}
\begin{proof}
Again we represent conjugacy classes of cocharacters of $L$ by
$L$-dominant elements of $X_*(T)$. An element
$$
\mu'=\big((\mu'_{11},\dotsc,\mu'_{15}),(\mu'_{21},\dotsc,\mu'_{25})\big)\in
X_*(T)
$$
is in the $G$-conjugacy class of $\mu$ if and only if it satisfies $\mu'_{ij}\in
\{0,1\}$ and $\sum_{j=1}^5\mu'_{1j}=2$ and $\sum_{j=1}^5\mu'_{2j}=3$. From
$[b]\in B(L,\{\mu'\})$ we  obtain
$\mu'_{11}+\mu'_{21}+\mu'_{12}+\mu'_{22}=1$ and the sum of the remaining
coordinates is $4$. This leaves
two elements of $I_{b, {\{ \mu \}}, L}$, namely
$$I_{b, {\{ \mu \}}, L}=\big\{\big((0^{(2)},0, 1^{(2)}), (0,1, 0,1^{(2)})\big),
\big((0,1,0^{(2)},1),(0^{(2)},1^{(3)})\big)\big\}.
$$
The dimensions of the corresponding local Shimura varieties can be
computed as sums of dimensions of the corresponding RZ spaces for the groups
$\Res_{F/\mathbb{Q}_p}\GL_2$, resp.~ $\Res_{F/\mathbb{Q}_p}\GL_3$. They are equal to $\langle \mu', 2\rho_L\rangle$, where $\rho_L$ is the
half-sum of the positive roots of $T$ in $L$. Thus we obtain in the two
cases that the dimensions of the associated local Shimura varieties are 5 and 3, respectively.
\end{proof}
\end{example}

Using the sets  $I_{b, \{\mu\}, L}$ we may now formulate the conjecture. For this we make the following additional assumption:
\begin{equation}\label{condstar}
\text{  \it $J_b$ is an inner form of a Levi subgroup contained in $L$. }
\end{equation} 
Note that this assumption is indeed a restriction. An easy example is $b=1$, in which case $J=G$, but without \eqref{condstar}, there is no restriction on $L$. The condition implies that there exists a unique parabolic $P$ with Levi subgroup
$L$ such that $P \otimes_F \breve{F}$ contains the parabolic subgroup $P_b$ of $G \otimes_F \breve{F}$
corresponding to the slope filtration for $b \sigma$, i.e.,
\begin{equation*}
 P_b (\breve{F}) = \{ g \in G (\breve{F}) \mid \lim_{t \rightarrow 0} \, g  \nu_b (t) g^{-1} \, \text{ exists}  \} \, .
\end{equation*}

For 
${\{ \mu' \}_L} \in I_{b, \{\mu\}, L}$, we denote by $H^{\bullet} (( L, [b]_L, {\{ \mu' \}_L} ))[{\rho}]$
the corresponding element of ${\rm Groth} (L(F) \times W_{E_{\{\mu'\}}} )$. Here, as before, $\rho$ is an irreducible
admissible representation of $J_b (F)$. Note that, even if the local Shimura variety 
$\big(\BM (G, [b], {\{\mu\}})^K\big)_K$ comes from a RZ space, this is not necessarily true of 
$\big(\BM (L, [b], {\{\mu'\}})^{K_L}\big)_{K_L}$, cf. Examples \ref{secexlsh}, (iii).
\begin{conjecture}[HV-Conjecture]\label{HVconj}
Under assumption \eqref{condstar}  there is an equality in ${\rm Groth} (G (F) \times W_{E_{\{\mu\}}} )$,
\begin{equation*}
 H^{\bullet} ((G, [b], {\{\mu\}})) [\rho] = {\rm Ind}^{G(F)}_{P(F)} \big(\sum_{{\{\mu'\}_L} \in I_{b, \{\mu\}, L}}
H^{\bullet} ((L, [b]_L, {\{\mu'\}_L}))[\rho]\big).
\end{equation*}
\end{conjecture}
Note that $W_{E_{\{\mu\}}} \supset \bigcup_{\{\mu'\}_L\in I_{b, \{\mu\}, L}} W_{E_{\{\mu'\}_L}}$, and that the actions of
$W_{E_{\{\mu'\}_L}}$ on the individual summands occuring on the RHS extend in a natural way to an action of
$W_{E_{\{\mu\}}}$ on the total sum, and hence on the induced representation.

We note the following consequence of the HV-conjecture.
\begin{corollary}
 [of conjecture \ref{HVconj}]\label{corhv} Assume the HV-conjecture for $(G, [b], {\{\mu\}})$ and $\rho$,  and let
$L \subsetneq G$. Let $\pi$ be a supercuspidal representation of $G(F)$. Then the $\pi$-isotypical component of $H^{\bullet} ((G, b, \mu))[{\rho}]$ is zero.
\end{corollary}
\begin{proof}
 Indeed, the $\pi$-isotypical component of any properly induced representation vanishes.
\end{proof} 
\begin{remark}The same statement should be true for an arbitrary irreducible admissible representation $\pi$ 
 with a discrete
$L$-parameter. 
\end{remark}
\begin{remark}
 Let $L'$ be a Levi subgroup containing $L$, and let $P'$ be the parabolic subgroup with Levi factor
$L'$ containing $P$. Then the set $I^G_{b, {\{ \mu \}}, L}$ of (\ref{defofI}) decomposes into a disjoint sum of sets
$I^{L'}_{b, {\{ \mu' \}}, L}$, where $\{\mu'\}$ ranges over $I^G_{b, {\{ \mu \}}, L'}$. Assuming the HV-conjecture
for $(L',L)$, we have
\begin{equation*}
 H^{\bullet} (\BM (L', [b]_{L'}, {\{\mu'\}}))[{\rho}] = {\rm Ind}^{L'}_{P \cap L'} \big(\sum_{\{\mu''\}\in I^{L'}_{b, \{\mu'\}, L}} H^{\bullet}
(\BM (L, [b]_{L}, {\{\mu''\}}_L)[{\rho}]\big).
\end{equation*}
Summing both sides over $I^G_{b, \{\mu\}, L'}$, we obtain
\begin{equation*}
\begin{aligned}
 \sum_{\{\mu'\} \in I^G_{b, \{\mu\}, L'}} H^{\bullet} (\BM (L', [b]_{L'}, {\{\mu'\}})[{\rho}]& = \\
={\rm Ind}^{L'}_{P \cap L'} &\big(\sum_{\{\mu''\} \in I^G_{b, \{\mu\}, L}} H^{\bullet} (\BM (L, [b]_{L}, {\{\mu''\}}_L)
[{\rho}]\big) .
\end{aligned}
\end{equation*}
Hence, assuming the HV-conjecture for $(G,L')$, we obtain
\begin{equation*}
 H^{\bullet} (\BM (G, [b], {\{\mu\}}))[{\rho}] = {\rm Ind}^{G}_{P'} ({\rm LHS}) \, ,
\end{equation*}
which we may identify, by the transitivity of induction, with
\begin{equation*}
 {\rm Ind}^{G}_{P} \big(\sum_{\{\mu'\} \in I^G_{b, \{\mu\}, L}} H^{\bullet} 
(\BM (L, [b]_{L}, {\{\mu'\}}))[{\rho}] \big) \, ,
\end{equation*}
as predicted by the HV-conjecture for $(G, L)$.
\end{remark}

\subsection {Known results} In the remainder of this section, we list some cases where the HV-conjecture is known to hold. The earliest result in this direction is due to  Boyer in the equal characteristic case (for  elliptic $\mathcal{D}$-modules), \cite{B}. Later the direct analog in mixed characteristic of Boyer's technique has been used as one important ingredient in the proof by Harris/Taylor of the local Langlands correspondence for $\GL_n$. The key to their proof is  that in their case ($G=\GL_n, \{\mu\}=(1, 0^{(n-1)})$), if $[b]$ is non-basic, then the universal $p$-divisible group over the corresponding RZ tower is an extension of an \'etale $p$-divisible group by a bi-infinitesimal one. The following  is the most general result in this direction. 
\begin{theorem}[Mantovan]\label{thmmant}
Let $\mathcal{D}_{\BZ_p}$ be an unramified simple integral RZ-datum (comp. Remark \ref{unramifiedRZ}), and let $([b],\{\mu\})$ be the associated neutral acceptable pair. Let $L$ be the Levi subgroup of a parabolic $P$ over $\BQ_p$ such that $P$  is standard with respect to the choice of a Borel  subgroup  $B$ over $ \BQ_p$. Let $T$ be a maximal torus over $\BQ_p$ contained in $B\cap L$, and denote by $\mu_{\rm dom}$ the $B$-dominant representative of $\{\mu\}$ in $X_*(T)$. Assume that $b\in L(\breve \BQ_p)$ is such  that $[b]_L\in B(L,\{\mu_{\dom}\}_L)$ and such that $B\otimes_{\BQ_p}\breve\BQ_p\subset P_b\subset P\otimes_{\BQ_p}\breve\BQ_p$.  Then
\begin{altenumerate}
\item the set $I_{b, \{\mu\}, L}$ consists
of the single element $\mu_{\rm dom}$.
\item the HV-conjecture holds for $(G, b, \{\mu\})$ and $L$, for any $\rho$.
\end{altenumerate}
\begin{remark}
Let $\mathcal{D}_{\BZ_p}$ be an unramified simple integral RZ-datum and let $([b],\{\mu\})$ be the associated neutral acceptable pair. Assume that $[b]$ is non-basic and that $([b],\{\mu\})$  is  HN-reducible in the sense of Definition \ref{defhnregular}. Then by \cite{CKV}, Theorem 2.5.6, there is a proper standard parabolic subgroup $P$ of $G$ with Levi factor $L$ and a representative $b\in L(\breve \BQ_p)$ as demanded in the theorem.

In the special case where $G$ is split (i.e., ~equal to $\GL_{n}$ or ${\rm GSp}_{2n}$), the RZ space parametrizes $p$-divisible groups of a fixed isogeny class (and possibly a polarization). The smallest parabolic subgroup as in the theorem corresponds in this case to the filtration of these $p$-divisible groups into their multiplicative, bi-infinitesimal, and \'etale parts. For general $G$ this is still one possible choice for $P$, but in general not the smallest one. Note that this choice for $P$ is in general (even for $G=\GL_n$) much larger than the smallest one allowed in the conjecture. Indeed, for $G=\GL_n$, the smallest parabolic allowed in the conjecture corresponds to the slope filtration of the $p$-divisible groups (i.e.,~such that the subquotients are isoclinic).

In Example \ref{exIbig}, the condition in the theorem is not satisfied.
Indeed, in this case the only proper standard Levi subgroup containing an
element of $[b]$ and such that $B\otimes_{\BQ_p}\breve\BQ_p\subset
P_b\subset P\otimes_{\BQ_p}\breve\BQ_p$ is the subgroup $L$ considered in
the example. Then the dominant representative $\mu$ of $\{\mu\}$ is not
contained in $I_{b,\{\mu\},L}$, hence $[b]_L\notin B(L,\{\mu\}_L)$.
\end{remark}

\end{theorem}
\begin{proof}[Proof of Theorem \ref{thmmant}] (i) If $\{\mu'\}\in I^G_{b,\{\mu\},L}$ then $\mu'_{\dom}=\mu_{\dom}$. Hence $\mu'_{L-\dom}$ (the unique $L$-dominant representative of $\{\mu'\}_L$ in $X_*(T)$) is in the $W$-orbit of $\mu_{\dom}$. Thus $\mu_{\dom}-\mu'_{L-\dom}\in \pi_1(L)$ is a non-negative linear combination of the roots of $T$ in the unipotent radical of $P$. On the other hand, $\mu_{\dom}-\mu'_{L-\dom}\in \pi_1(L)$ maps to $0$ in $\pi_1(L)_{\Gamma}$,  as it is the difference of two elements mapping to $\kappa_L(b)$. Since $P$ and its unipotent radical are defined over $F$ this implies that the images of $\mu_{\dom}$ and $\mu'_{L-\dom}$ in $\pi_1(L)$ agree. As both $\mu_{\dom}$ and $\mu'_{L-\dom}$ are $L$-dominant and minuscule,  they thus coincide by (\ref{minuscrep}). 

(ii) This is proved in \cite{M2}, Corollary 1.5.
\end{proof}
\begin{remark}
The first assertion was first claimed (without proof) in \cite{H} in a more general context than that where it holds, compare the remark in Example \ref{exIbig}. It is implicitly used in Mantovan's proof of (ii).

In \cite{M2}, assertion (ii) is proved in those cases where locally the universal $p$-divisible group over the RZ space has a slope filtration related to the Levi subgroup $L$. The existence of this slope filtration is proved in \cite{mv}. To be more precise, in \cite{mv} an additional condition is imposed that is, however, not necessary. This additional unnecessary condition also appears in the final result in \cite{M2}. Theorem \ref{thmmant} and the required slope filtration theorem in the above formulation are first stated and proved by Shen \cite{She}. 
\end{remark}

In \cite{Shin}, 8.3, Shin sketches the proof of  some further special cases of the conjecture for $p$-divisible groups of low height, but where one no longer has a slope filtration. 

 For level structures with respect to $\CG(\BZ_p)$, Mantovan has a more general result using a study of the integral RZ spaces.
\begin{theorem}[Mantovan]\label{thmmantk}
Let $\mathcal{D}_{\BZ_p}$ be an unramified simple integral RZ-datum.  Let $([b],\{\mu\})$ be the associated neutral acceptable pair and $\CG$ the corresponding parahoric group scheme, and let $L$ and $b\in L(\breve\BQ_p)$ be as in Conjecture \ref{HVconj}. Then 
\begin{equation*}
\label{eqmantk}H^{\bullet} ((G, [b], {\{\mu\}})) [\rho]^{K_0} = {\rm Ind}^{G(\mathbb{Q}_p)}_{P(\mathbb{Q}_p)} \big(\sum_{{\{\mu'\}_L} \in I_{b, \{\mu\}, L}}
H^{\bullet} ((L, [b]_L, {\{\mu'\}_L}))[\rho]\big)^{K_0}
\end{equation*} 
where $K_0=\CG(\BZ_p)$.
\end{theorem}
\begin{remark}
By definition of the action of $G(\mathbb{Q}_p)$ on the tower and its cohomology, we can rewrite the left hand side of the identity in Theorem \ref{thmmantk} as 
$$H^{\bullet} ((G, [b], {\{\mu\}})) [\rho]^{K_0} =\sum_{i,j\geq 0}(-1)^{i+j}H^{i,j}(\mathbb{M}(G,[b],\{\mu\})^{K_0})[\rho].$$
Similarly, using that $P(\mathbb{Q}_p)K_0=G(\mathbb{Q}_p)$, the right hand side of the identity is equal to
$$\sum_{{\{\mu'\}_L} \in I_{b, \{\mu\}, L}}\sum_{i,j\geq 0}(-1)^{i+j}H^{i,j}(\mathbb{M}(L,[b]_L,\{\mu'\}_L)^{L\cap K_0}))[\rho].$$
\end{remark}

\begin{proof}[Proof of Theorem \ref{thmmantk}]
This is \cite{M2}, Theorem 9.3, except for the fact that Mantovan uses Harris' incorrect proposition that $|I^G_{b,\{\mu\},L}|=1$. Making the necessary modifications to correct this, her proof shows the above assertion.
\end{proof}

It might be easier to prove the consequence Corollary \ref{corhv} of the HV conjecture, rather than the conjecture itself. Perhaps the case where $G=\GL_n$ is doable. This would also constitute some justification for the correct choice of the cohomology theory in this context, cf. the remarks at the beginning of this section.

%---------------------------------------------------

\bigskip
\obeylines
Mathematisches Institut der Universit\"at Bonn  
Endenicher Allee 60 
53115 Bonn, Germany.
email: rapoport@math.uni-bonn.de

\bigskip
\obeylines
Fakult\"at f\"ur Mathematik der Technischen Universit\"at M\"unchen -- M11
Boltzmannstr. 3 
85748 Garching, Germany
email: viehmann@ma.tum.de

%xxxxxxxxxxxxxxxxxxxxxxxxxxxxxxxxxxxxxxxxxxxxxxxxxxxxxxxxxxxxxxxxxxxxxxxxxxxxxxxxxxxxxxxxxxxxxxxxxxxxxxxxxxxxxxxxxxxxxxxxxxxxxxxxxxxxxxxxxxxxxxxxxxxxxxxxx
\end{document}